\documentclass[11pt,a4paper]{amsart}
\usepackage{amsfonts,amscd,amssymb,amsmath,amsthm,hyperref,mathrsfs,multirow,xcolor,lscape,longtable,pbox,lipsum}
\usepackage[thinlines]{easytable}
\usepackage{microtype}
\usepackage{todonotes}
\newtheorem{thm}{Theorem}

\newtheorem{lemma}[thm]{Lemma}
\newtheorem{prop}[thm]{Proposition}
\newtheorem{cor}[thm]{Corollary}
\newtheorem{rmk}[thm]{Remark}

\usepackage[all,cmtip]{xy}
\usepackage{tikz}
\usetikzlibrary{arrows} 
\usetikzlibrary{matrix}
\usetikzlibrary{calc}

\newcommand{\R}{\mathbb{R}}
\newcommand{\Z}{\mathbb{Z}}
\newcommand{\Q}{\mathbb{Q}}
\newcommand{\F}{\mathbb{F}}

\newcommand{\G}{\Gamma}

\newcommand{\T}{\mathrm{T}}

\newcommand{\h}{\mathbb{H}}

\newcommand{\m}{\mathcal{M}}

\newcommand{\rG}{\mathrm{G}}
\newcommand{\rS}{\mathrm{S}}

\newcommand{\rP}{\mathrm{P}}

\newcommand{\tmt}[4]{\left({#1\atop #3}{#2\atop #4}\right)}

\newcommand{\eqr}[1]{\mbox{(\ref{eq:#1})}}
\newcommand{\ie}{i.e.\ }
\newcommand{\mr}[1]{\mathrm{#1}}
\newcommand{\C}{\mathbb{C}}

\newcommand{\mc}[1]{\mathcal{#1}}

\begin{document}

\title[Euler Characteristics of Congruence Subgroups]{Euler Characteristics of a Family of Congruence Subgroups of $GL_m(\Z)$} 
\date{\today}
\author{Ivan Horozov}
\address{\tiny{City University of New York}}
\email{ivan.horozov@bcc.cuny.edu}

\subjclass[2010]{11F75; 11F70; 11F22; 11F06}  
\keywords{Arithmetic Groups, Modular Group, Congruence Subgroups, Automorphic Forms, Boundary Cohomology, Euler Characteristic, Eisenstein Cohomology, Cuspidal Cohomology}
\begin{abstract}
The congruence subgroups $\Gamma_1(m,p)$ that we consider here are subgroups of $GL_m(\Z)$ that fix the vector $(0,\dots,0,1) \mod p$, where $p\geq 5$ is a prime. We present a method and many computations of homological Euler characteristics of $GL_m(\Z)$ and $\Gamma_1(m,p)$ with coefficients in any highest weight representation $V$. By homological Euler characteristics we mean the alternating dimensions of cohomology of the group with coefficient in $V$. We compute the homological Euler characteristics for 
$\Gamma_1(2,p)$, 
and $\Gamma_1(3,p)$
with coefficients in any finite dimensional highest weight representation.
Also we compute the  homological Euler characteristics for of
$\Gamma_1(4,p)$
and
$\Gamma_1(5,p)$ with coefficients in the trivial and the determinant representations.
We give application to cohomology of $\Gamma_1(3,p)$ with trivial and with determinant representation.
We also give an alternative method for computing the cohomology of $GL_4(\Z)$ compared to \cite{GL4}.
The methods in this paper are a continuation of result from \cite{Thesis, EulerChar}. 
\end{abstract}

\maketitle
\tableofcontents




\section{Introduction}\label{intro}
Homological Euler characteristic is a type of a trace formula. It is defined by the alternating sum of dimensions of cohomology arithmetic groups.
Besides being important in its own right, it turns out to be very useful in computation of cohomology of arithmetic groups, see \cite{GL4,BHHM, BHM}.
In this paper we will focus on the arithmetic $\Gamma_1(m,p)$. It is a subgroup of $GL_m(\Z)$ which fixes the vector $(0,\dots,0,1)$ modulo a prime $p$. We assume that $p\geq 5$. 

For convenience of the reader first we consider the $2\times2$ case in great details. There are several reasons for this approach. First, most of the results in this case are known. Second, it introduces the notation for higher dimension. And most importantly, the generalizations to higher dimensions to the $m\times m$ case look similar to the $2\times 2$-case.

The approach here is similar to \cite{Thesis, EulerChar}. Here we focus on the arithmetic groups $\Gamma_1(m,p)$. 
The basic part of the paper is a synthesized version of a portion of my thesis. In addition to that, we present new formulas for 
the homological Euler characteristics of $\Gamma_1(2,p)$,  $\Gamma_1(3,p)$, $\Gamma_1(4,p)$ and  $\Gamma_1(5,p)$. It also corrects one formula for the homological Euler characteristics of $\Gamma_1(4,p)$ with trivial coefficients. The error was simply in the addition of (many) fractions.
We give applications to Eisenstein cohomology of $\Gamma_1(m,p)$ for $m=2,3$. For the $2\times2$-case, we compute the Eisenstein and the cuspidal cohomology of  $\Gamma_1(2,p)$ with coefficients in any finite dimensional highest weight representation of $GL_2$. 
Notice, $\Gamma_1(2,p)$ is not $\Gamma_1(p)$. The relation is $\Gamma_1(p)=\Gamma_1(2,p) \cap SL_2(\Z)$. However, the corresponding result for $\Gamma_1(p)$ follow directly from the ones about $\Gamma_1(2,p)$, just by summing up two terms. In particular, we obtain the dimension of holomorphic cusp forms for $\Gamma_1(p)$ for any integer weight $n\geq 2$. 

As an application, we compute the cohomology of $\Gamma_1(3,p)$ with trivial coefficients and with determinant coefficients and compare them to the cohomology of $\Gamma_1(3,p)\cap SL_3(\Z)$ considered in \cite{LSch}. This computation is needed for Goncharov's approach for computing the dimensions of spaces of motivic multiple zeta values and polylogarithms. Besides the computation, I give several explanations for the differences the previously obtained formula with the current one, explaining also the source of the error.

There is also, one more application on the homological Euler characteristic of $GL_4(\Z)$ with trivial coefficients and with determinant coefficient and its application to cohomology The method is very similar to the previous application, in fact simpler. However, it needed for a current paper by Goncharov again on the topic of motivic multiple zeta values. That result was obtained in \cite{GL4}  using a different method.

\subsection{Basic properties of Euler characteristics}
Before the state the key results of the paper, we need to introduce some notation. Also, we need to explain some of the properties in order for the formulations make sense.

The homological Euler characteristic $\chi_h$ of a group $\Gamma$ with coefficients in a representation $V$ is defined by 
\begin{equation}\label{eq:hec}
\chi_h(\Gamma,V)=\sum_{i=0}^{\infty}\, (-1)^{i} \, \mr{dim} \, H^{i}(\Gamma,V). 
\end{equation}
For details on the above formula  see~\cite{B1,B2,Se}. We recall the definition of orbifold Euler characteristic. If $\Gamma$ is torsion free, then the orbifold Euler characteristic is defined as $\chi_{orb}(\Gamma) = \chi_h(\Gamma)$. If $\Gamma$ has torsion elements and admits a finite index torsion free subgroup $\Gamma'$, then the orbifold Euler characteristic of $\Gamma$ is given by 
\begin{equation}\label{eq:orbeuler}
\chi_{orb}(\Gamma) = \frac{1}{[\Gamma : \Gamma']} \chi_h(\Gamma')\,.
\end{equation} 
One important fact is that, following Minkowski, every arithmetic group of rank greater than one contains a torsion free finite index subgroup and therefore the concept of orbifold Euler characteristic is well defined in this setting. If $\Gamma$ has torsion elements then we make use of the following  formula discovered by Wall in~\cite{W}.
\begin{equation}\label{eq:hecT}
\chi_h(\Gamma,V )=\sum_{(T)} \chi_{orb}(C(T)) Tr(T^{-1}| V),
\end{equation}
where $Tr(T^{-1}|V)$ is the trace of the element $T^{-1}$ acting on the representation $V$

 Otherwise, we use the formula described in equation~\eqr{hec}. The sum runs over all the conjugacy classes in $\Gamma$ of its torsion elements $T$, denoted by $(T)$, and $C(T)$ denotes the centralizer of $T$ in $\Gamma$. The identity element $I$ is also considered as a torsion element. Its centralizer is the entire group. 
 
 From now on, the orbifold Euler characteristic $\chi_{orb}$ will be simply denoted by $\chi$. It has the following properties.
 \begin{enumerate}
\item  If $\Gamma$  is finitely generated  torsion free group  then $\chi(\Gamma)$ is defined as $\chi_h(\Gamma,\Q).$
\item If $\Gamma$ is finite of order $\left| \Gamma \right|$ then $\chi(\Gamma)=\frac{1}{\left| \Gamma \right|}$.
\item Let $\Gamma$, $\Gamma_1$ and  $\Gamma_2$ be groups such that 
$ 1 \longrightarrow \Gamma_1 \longrightarrow \Gamma \longrightarrow \Gamma_2 \longrightarrow 1$ is exact then $\chi(\Gamma)= \chi(\Gamma_1) \chi(\Gamma_2)$.
 \end{enumerate}

For the homological Euler characteristics, we need a few more definitions.

Let us recall briefly the notation $R(f)$. Let $f_1=\prod_i(x - \alpha_i)$ and $f_2=\prod_j(x-\beta_j)$ be two polynomials. Then by the resultant of $f_1$ and $f_2$, we mean $R(f_1,f_2)=\prod_{i,j}(\alpha_i-\beta_j)$. If the characteristic polynomial $f$ is a power of an irreducible polynomial then we define $R(f)=1$. Let $f=f_1 f_2\dots f_d$, where each  $f_i$ is a power of an irreducible polynomial over $\Q$ and they are relatively prime pairwise. Then, we define $R(f)=\prod_{i<j} R(f_i,f_j)$. We define $R(A)$ for a matrix $A$ to be $R(f)$, where $f$ is the characteristic polynomial of the matrix $A$.

Let us denote by
\[ 
T_3 = \left( \begin{array}{rrr}
0 & 1  \\
-1 & -1 \\ 
\end{array}  \right),
T_4 = \left( \begin{array}{rrr}
0 & 1 \\
-1 & 0 \\ 
\end{array}  \right) \mbox{ and }
T_6 = \left( \begin{array}{rrr}
0 & -1  \\
1 & 1   \\ 
\end{array}  \right)
\]
representatives of $3$-, $4$- and $6$-torsion elements of the group $GL_2(\Z)$.

Then following ~\cite{Thesis,EulerChar}, for  the arithmetic group $GL_m(\Z)$ with a representation  $V$, one has the following computationally effective formula for the homological Euler characteristics: 
\begin{thm}
 \begin{equation}\label{eq:hnthor}
 \chi_{h}(GL_m(\Z), V) = \sum_{A} R(A) \chi(C(A)) Tr(A| V)\,,
 \end{equation}
The summation is over all possible block-diagonal torsion matrices $A \in \Gamma$ satisfying the following conditions:
\begin{itemize}
\item The blocks in the diagonal belong to the set $\left\{1, -1,  T_3, T_4, T_6\right\}$.
\item The blocks $T_3, T_4$ and $T_6$ appear at most once and the blocks  $1$ and $-1$ appear at most twice.
\item A change in the order of the blocks in the diagonal does not count as a different element.
\end{itemize}
\end{thm}

This Theorem is based on the following two Lemmas.
The first lemma is used to determine which diagonal blocks are enough to consider.
\begin{lemma}
\label{tor type}
If $A$ is a torsion element of $GL_m(\Z)$ and if the Euler characteristic of its centralizer is not zero, then its eigenvalues are $1,-1,i,-1,w,w^2,w^4$ or $w^5$, where $w$ is primitive $6^{th}$-root of unity. The eigenvalues $1$ and $-1$ must have multiplicity at most two and the rest multiplicity one.
\end{lemma}
\proof
This is a property of the algebraic group of centralizer of $A$, $C(A)$. We have that $C(A)$ is a product of algebraic groups of the type $GL_{n}$ over $\Q(\mu)$. 
From ~\cite{Ha1}, it follows that the orbifold Euler characteristic vanishes in the following cases $\chi(GL_n(\mathbb{Z})) = 0$ for $n>2$.
If $\mu$ is a 3rd, 4th or 6th root of unity then $\chi(GL_n(\mathbb{Z}[\mu])) = 0$ for $n>1$. And if $\mu$ a root of unity different from $1$, $-1$, 3rd, 4th and 6th root of unity, then $\chi(GL_n(\mathbb{Z}[\mu])) = 0$ for all $n$, including $n=1$.
\qed

For block-upper triangular matrices, we need to simplify them by conjugation to obtain block diagonal ones. This is done using the second lemma, which explain the source of the resultants that define $R(A)$.

\begin{lemma}
\label{P}
Let $A_{11}$ and $A_{22}$ be square torsion matrices of finite order possibly of different size.
Let $\alpha_1,\dots,\alpha_m$ be the eigenvalues of $A_{11}$ and let $\beta_1,\dots,\beta_n$ be the eigenvalues of $A_{22}$. Assume that $A_{11}$ and $A_{22}$ are torsion elements.
Let $P$ be the linear map from the space of $m\times n$ matrices over $\Z$ to itself
$P(X)=XA_{11}-A_{22}X$. Then, all the eigenvalues of $P$ are $\alpha_i-\beta_j$ for $i=1,\dots,m$ and $j=1,\dots,n$.
\end{lemma}
\proof Since $A_{11}$ and $A_{22}$ are torsion matrices, they are simultaneously diagonalizable for the following two reasons. First, each of them is a torsion matrix therefore diagonalizable. Second, they are automorphisms of different spaces. One of them is $m$-dimensional and the other is $n$-dimensional. Tensor the linear transformation $P$ with the complex numbers $\C$. Rewrite the linear transform $P$ over $\C$ in a basis where $A_{11}$ and $A_{22}$ are diagonal. In that basis, each entry of $P$ say $(i,j)$-entry is an eigenvector with eigenvalue $\alpha_i-\beta_j$.

As a consequence of the first lemma, we have the following.
\begin{cor}
For any arithmetic group $\Gamma$ commensurable to $GL_m(\Z)$ with $m>10$ and any finite dimensional representation of $\Gamma$, we have
\[\chi_{h}(\Gamma, V) = 0.\]
\end{cor}

\subsection{Key results of the paper}
For $m\leq 10$ and  for the arithmetic groups $\Gamma_1(m,p)$, we have the following formula for the homological Euler characteristics.
\begin{thm}
\label{chi1}
For a prime $p\geq 5$ we have the following formula for the homological Euler characteristic,
\begin{align}
\chi_h(\Gamma_1(m,p),V) = & \varphi(p) \sum_{A=[1,A_1]} R(A)\chi(C(A))Tr(A|V) \nonumber\\
& + \varphi_2(p) \sum_{A=[I_2,A_2]} R(A)\chi(C(A))Tr(A|V).
\end{align}
In this summation, $A_1$ and $A_2$ runs through all $(m-1)\times(m-1)$-matrices and $(m-2)\times(m-2)$-matrices respectively in block-diagonal form so that the blocks (without repetition) are from the set
$\{-1, -I_2, T_3, T_4, T_6\}$.
The centralizers are computed inside $GL_m(\Z)$,
Also, $\varphi(p)=p-1$ is the Euler phi function and $\varphi_2(p)=p^2-1$ is a generalization of it.
\end{thm}

\begin{thm}
\label{chi1}
For a prime $p\geq 5$ we have the following formula for the homological Euler characteristic,
\[\chi_h(\Gamma_1(2,p),V) =  \frac{p-1}{2} Tr([1,-1]|V) + \frac{p^2-1}{24}\dim(V).\]
\end{thm}

\begin{thm} Let $S^n$ denote the $n$-symmetric power of the standard $2$-dimensional representation of $GL_2$.
Let $S_{n+2}(\Gamma_1(p))$ be the space of holomorphic cusp form of weight $n+2$ for the group $\Gamma_1(p)$.
Then
 \[\dim S_{2}(\Gamma_1(p))=\dim H^1_{cusp}(\Gamma_1(2,p),\C)=\dim H^1_{cusp}(\Gamma_1(2,p),det)\]
 and
\[\dim S_{2}(\Gamma_1(p))=1+\frac{1}{24}(p^2-1) -\frac{1}{2}(p-1).
\]
For $n>0$, either even or odd,
\[\dim S_{n+2}(\Gamma_1(p))=\dim H^1_{cusp}(\Gamma_1(2,p),S^n)=\dim H^1_{cusp}(\Gamma_1(2,p),S^n\otimes det)\]
and
\[\dim S_{n+2}(\Gamma_1(p)) =\frac{n+1}{24}(p^2-1)-\frac{1}{2}(p-1).\]
\end{thm}

\begin{thm} 
The dimensions Eisenstein cohomology of $\G_1(2,p)$, is given below.

(a) $\dim H^0_{Eis}(\Gamma_1(2,p),\Q)=1$

(b) $\dim H^1_{Eis}(\Gamma_1(2,p),det)
=\varphi(p)-1$

(c) For even $n$ and $n>0$, we have
$
\dim H^1_{Eis}(\Gamma_1(2,p),S^n)=0
$

(d) For even $n$ and $n>0$, we have
$\dim H^1_{Eis}(\Gamma_1(2,p),S^n\otimes det)=p-1
$

(e) For odd $n$, we have
$
\dim H^1_{Eis}(\Gamma_1(2,p),S^n)=\frac{1}{2}(p-1)$

(g) For odd $n$, we have
$
\dim H^1_{Eis}(\Gamma_1(2,p),
S^n \otimes det)=\frac{1}{2}(p-1)
$
\end{thm}

\begin{thm} (Homological Euler characteristics of $\Gamma_1(3,p)$)
Let $V=V_\lambda$ be a highest weight representation of $GL_3$. Then the homological Euler characteristic of $\Gamma_1(3,p)$ with coefficients in $V$ is
\[
\chi_h(\Gamma_1(3,p),V)=\frac{(p-1)}{2}\chi_h(GL_3(\Z),V)-\frac{(p^2-1)}{12}Tr([-1,1,1]|V),
\]
Moreover,  $Tr([-1,1,1]|V_\lambda)$ is bounded and periodic with respect to the highest weight $\lambda$.
\end{thm}

\begin{thm} (Cohomology of $\Gamma_1(3,p)$)
(a) The cohomology of $\Gamma_1(3,p)$ with trivial coefficients is
\[
H^i_{Eis}(\Gamma_1(3,p),\C)=
\left\{
\begin{tabular}{ll}
$\C$									& $i=0$\\
$0$									& $i=1$\\
$0$									& $i=2$\\
$2S_2(\Gamma_1(p))+(-1+\frac{p-1}{2})\C$	& $i=3$.
\end{tabular}
\right.
\]

The dimension of $H^3$ is
\[
\dim H^3(\Gamma_1(3,p),\C)=
1 + \frac{2}{24}(p^2-1) + \frac{1}{2}(p-1).
\]

(b) The cohomology of $\Gamma_1(3,p)$ with determinant coefficients is
trivial except for 
\[H_{Eis}^2(\G_1(3,p),det)=
\varphi(p)\C
+ S_3(\Gamma_1(p)).
\]
Its dimension is
\[
\dim H_{Eis}^2(\G_1(3,p),det)= \frac{2}{24}(p^2-1)-\frac{1}{2}(p-1).
\]
\end{thm}

\begin{thm} (Cohomology of $GL_4(\Z)$)

(a) The cohomology of $GL_4(\Z)$ with trivial coefficients is zero except for $H^0$ which is $\C$.

(b) The cohomology of $GL_4(\Z)$ with trivial determinant coefficients is concentrated in degree $3$.
\[\dim H^3(GL_4(\Z),det)=1.\]
\end{thm}

As a consequence, we obtain the following.
\begin{thm}
\label{Gamma1(4,p)}
(a)The Euler characteristic of $\Gamma_1(4,p)$ with trivial coefficients is
\[\chi_h(\Gamma_1(4,p),\C)=(p-1)-\frac{1}{12}(p^2-1).\]

(b)The Euler characteristic of $\Gamma_1(4,p)$ with determinant coefficients is
\[\chi_h(\Gamma_1(4,p),det)=-(p-1)-\frac{1}{12}(p^2-1).\]
\end{thm}

\begin{thm} (Homological Euler characteristics of $\Gamma_1(5,p)$)
Let $V=V_\lambda$ be a finite dimensional highest weight representation of $GL_5$. Let $\Sigma^{(-1)}(V)$ be the summation of $R(A)\chi(C(A))Tr(A|V)$ over the following three conjugacy classes of torsion elements $A$ (inside $GL_5(\Z)$): 
$A=[-1,T_3,T_4]$, $A=[-1,T_6,T_4]$ and $A=[-1,T_3,T_6]$.
Then $\Sigma^{(-1)}(V_\lambda)$ is bounded and periodic with respect to the highest weight $\lambda$.
And the homological Euler characteristic of $\Gamma_1(5,p)$ with coefficients in $V$ is given by the following.

(a) If $-I\in GL_5(\Z)$ acts trivially on $V$ then
\[\chi_h(\Gamma_1(5,p),V)=((p-1)+(p^2-1)\chi_h(SL_5(\Z),V)-(p^2-1)\Sigma^{(-1)}(V).\]

(b) If $-I\in GL_5(\Z)$ acts as $-1$ on $V$ then
\[\chi_h(\Gamma_1(5,p),V)=((p-1)-(p^2-1)\chi_h(SL_5(\Z),V)-(p^2-1)\Sigma^{(-1)}(V).\]

(c) In particular for $V=\C$, we have
\[\chi_h(SL_5(\Z),\C)=0,\] 
$\Sigma^{(-1)}(\C)=\frac{1}{3}$,
and
 the homological Euler characteristics of $\Gamma_1(5,p)$ with trivial coefficients
is
\[\chi_h(\Gamma_1(5,p),\C)= -\frac{1}{3}(p^2-1).\]

(d) For $V=det$, we have $\Sigma^{(-1)}(det)=-\frac{1}{3}$ and
\[\chi_h(\Gamma_1(5,p),det)= \frac{1}{3}(p^2-1).\]

\end{thm}

\section{Euler Characteristic and Cohomology of  $GL_2(\Z)$ and $SL_2(\Z)$}

Consider the group $PSL(2,\Z) = SL_2(\Z)/ \{ \pm I_2\}$. For any subgroup $\G \subseteq SL_2(\Z)$ containing $-I_2$, we will denote by $\overline{\G}$ its projection to $PSL(2,\Z)$.

Consider the principal congruence subgroup $\overline{\G}(2)$. It is of index $6=|PSL(2,\F_2)|=|\overline{\G}(2)\backslash PSL(2,\Z)|$. Since $\overline{\G}(2)$ is torsion free, we have that ${\overline{\G}(2)}\backslash \h $ is biholomorphic to $\mathbb{P}^{1}-\{0,1,\infty \}$. Therefore, $$\chi(\overline{\G}(2))= \chi(\mathbb{P}^{1}-\{0,1,\infty \}) = \chi(\mathbb{P}^{1}) -3 =2 -3 =-1.$$ 
Using this, we immediately get $$\chi(\G(2))= \chi(\overline{\G}(2)) \chi(\{\pm I_2 \})= -1 \times \frac{1}{2}=-\frac{1}{2}.$$

Considering the following short exact sequence $$1\longrightarrow \G(2) \longrightarrow SL_2(\Z) \longrightarrow SL_2(\Z/2\Z)\longrightarrow 1.$$
We obtain $\chi(SL_2(\Z))= -\frac{1}{12}$ and $\chi(\overline{\Gamma}_0)= -\frac{1}{6}$. Similarly, the exact sequence $$1\longrightarrow SL_2(\Z) \rightarrow GL_2(\Z) \xrightarrow{det} \{ \pm I_2\} \longrightarrow 1,$$ where $det : GL_2(\Z) \longrightarrow \{ \pm I_2\}$ is simply the determinant map, gives $\chi(GL_2(\Z))= -\frac{1}{24}$.

Following equation~\eqr{hecT}, we compute the traces of all the torsion elements $T$ in $SL_2(\Z)$ and $GL_2(\Z)$ with respect to the highest weight representations $V_{m}$ and $V_{m_1, m_2}$ for $SL_2$  and $GL_2$ respectively. 

For any torsion element $T \in SL_2(\Z) $, we define 
\begin{equation*}\label{eq:hnt} H_{m}(T) := Tr(T^{-1}| V_{m})  = Tr(T^{-1} | Sym^{m} V)  =  \sum_{a+b =m}  \lambda_{1}^{a} \lambda_{2}^{b}\,. \end{equation*}
where $\lambda_1 $ and $\lambda_2$ are the two eigenvalues of $T$. From now on we simply denote the representative  of $n$ torsion element $T$  by its characteristic polynomial $\Phi_{n}$.  Therefore, 
{\begin{center}
\scriptsize\renewcommand{\arraystretch}{2}
\begin{longtable}{|c|c|c|c|c|c|c|}
\hline
Case & $T$ &  $\Phi_n$ &$C(T)$   &  $\chi(C(T))$ & $H_{m}(T)$    \\
\hline
\hline
A & $I_2$ & $\Phi_1^{2}$& $SL_2(\Z)$ &  $-\frac{1}{12}$ & $ m+1$   \\
\hline
B &$- I_2$ &$\Phi_2^{2}$  &$SL_2(\Z)$ &  $-\frac{1}{12}$ & $(-1)^{m}(m+1)$    \\
\hline
C &$T_3={\tmt{0}{1}{-1}{-1}} $ & $\Phi_3$  &$C_6$    & $\frac{1}{6}$  &$ (1, -1, 0)$\footnote{$(1,-1,0)$ signifies $H_{3k}(T)=1$, $H_{3k+1}(T)=-1$ and $H_{3k+2}(T)=0$.}   \\
\hline
D &$T_4={\tmt{0}{1}{-1}{0}}$ &   $\Phi_4$ &$C_4$    & $\frac{1}{4}$ & $(1, 0, -1, 0 )$     \\
\hline
E &$T_6{\tmt{0}{-1}{1}{1}}$ & $\Phi_6$   &$C_6$    & $\frac{1}{6}$ & $ (1, 1, 0 , -1, -1, 0)$   \\
\hline
F &${\tmt{1}{0}{0}{-1}}$ & $\Phi_1\Phi_2$   &$C_2\times C_2$    & $\frac{1}{4}$ & $ (1, 0)$   \\
\hline

\caption{Traces of torsion elements of $GL_2(\Z)$.}\label{tracetorsiojnsl2}
\end{longtable}
\end{center}
}

Now following equations~\eqr{hec} and~\eqr{hecT} 
\begin{eqnarray*}\label{eq:chisl2}
 \chi_{h}(GL_2(\Z), V_m )  
 = && -\frac{1}{24} H_{m}(\Phi_1^{2})  - \frac{1}{24} H_{m}(\Phi_2^{2})  + \frac{1}{6} H_{m}(\Phi_3) \\
 && + \frac{1}{4} H_{m}(\Phi_4)  + \frac{1}{6} H_{m}(\Phi_6)+\frac{1}{2}H_m(\Phi_1\Phi_2)\,.
  \end{eqnarray*}
We obtain the values of $\chi_{h}(GL_2(\Z), V_m) $ by computing each factor of the above equation~\eqr{chisl2} up to modulo 12. All these values can be found in the last column of the Table~\ref{eulersl2gl2} below.

Similarly, let us discuss the $\chi_{h}(GL_2(\Z), V_m\otimes det)$. 
It is essentially the same as $\chi_{h}(GL_2(\Z), V_m\otimes det)$. The only difference is that the last term is $-\frac{1}{2}H_m(\Phi_1\Phi_2)$ instead of $+\frac{1}{2}H_m(\Phi_1\Phi_2)$

\begin{eqnarray*}
 \chi_{h}(GL_2(\Z), V_m\otimes det )  
 = && -\frac{1}{24} H_{m}(\Phi_1^{2})  - \frac{1}{24} H_{m}(\Phi_2^{2})  + \frac{1}{6} H_{m}(\Phi_3) \\
 && + \frac{1}{4} H_{m}(\Phi_4)  + \frac{1}{6} H_{m}(\Phi_6)-\frac{1}{2}H_m(\Phi_1\Phi_2)\,.
  \end{eqnarray*}

For $SL_2(\Z)$, in general for $SL_m(\Z)$, we have 
\[H^q(SL_m(\Z),V)=H^q(GL_m(\Z),V)+H^q(GL_m(\Z),V\otimes det),\]
where $V$ is any finite dimensional representation. 

It is well known that 
\begin{equation}\label{eq:blah}
S_{m+2}=H^1_{cusp}(GL_2(\Z),V_m \otimes \mathbb{C})\subset H^1_!(GL_2(\Z),V_m\otimes \mathbb{C})\subset H^1(GL_2(\Z),V_m\otimes \mathbb{C}).
\end{equation}
One can show that in fact these inclusions are isomorphisms because $H^1(GL_2(\Z),\mathbb{C}) = 0$, and for $m > 0$ we have $H^0(GL_2(\Z),V_m) = H^2(GL_2(\Z),V_m) = 0$ and therefore
\begin{equation*}
\dim H^1(GL_2(\Z),V_m)= -\chi_h(GL_2(\Z),V_m)= \dim S_{m+2}.
\end{equation*}
Hence, we may conclude that for all $m$ $$H^1_{cusp}(GL_2(\Z),V_m \otimes \mathbb{C})= H^1(GL_2(\Z),V_m \otimes \mathbb{C}) =  H^1_!(GL_2(\Z),V_m) \otimes \mathbb{C}.$$ 

\begin{rmk}
Note that if we do not want to get into the transcendental aspects of the theory of cusp forms (Eichler-Shimura isomorphism) then we could get the dimension of $S_{m+2}$ by using the information given ~\cite{Thesis, EulerChar}.\end{rmk}

{\begin{center}
\scriptsize\renewcommand{\arraystretch}{2}
\begin{longtable}{|c|c|c|c|c|c|c|c|c|}
\hline
$m=12\ell +k $ & $ \chi_{h}(GL_2(\Z), V_m ) $  & $ \chi_{h}(GL_2(\Z), V_m \otimes det ) $ &  $ \chi_{h}(SL_2(\Z), V_m )$\\
\hline
\hline
$k =0 $   & $-\ell+1 $ &  $-\ell$  &  $-2\ell +1$ \\
\hline
$k =1 $  & 0   &  0 & 0 \\
\hline
$k =2 $ &  $-\ell$    & $ -\ell -1$  &  $-2\ell -1$ \\
\hline
$k =3 $ & 0  & 0   & 0  \\
\hline
$k =4$&  $-\ell$   & $ -\ell -1$   & $-2\ell -1$ \\
\hline
$k =5 $ & 0  & 0  & 0  \\
\hline
$k =6 $ & $-\ell$     & $ -\ell -1$   & $-2\ell -1$   \\
\hline
$k =7 $& 0 & 0  & 0 \\
\hline
$k =8$ &  $-\ell$    &  $ -\ell -1$  & $-2\ell -1$ \\
\hline
$k =9$ & 0 & 0  & 0 \\
\hline
$k =10$ &  $-\ell-1$     &  $ -\ell -2$   & $-2\ell -3$  \\
\hline
$k =11$ &  0 & 0  & 0 \\
\hline
\caption{Euler characteristics of $SL_2(\Z)$ and $GL_2(\Z)$.}\label{eulersl2gl2}
\end{longtable}
\end{center}
}


\subsection{Structure theory}\label{sl2}  

In this section, we will explain in the details the interesting and most fundamental cases, \ie $GL_2(\Z)$ and $SL_2(\Z)$.
We will discuss in detail the cohomology of $SL_2(\Z)$, $GL_2(\Z)$ and their family of two congruence subgroups, namely $\G_0(p):=\G_0(2,p)$ and $\G_1(p):=\G_1(2,p)$. 
They will be central in the computations of homological Euler characteristics and cohomology groups of higher rank arithmetic groups such as
$GL_m(\Z)$ and  $\G_{1}(m,p)$.

Let $\T$ be the maximal torus on $GL_2$ given by the group of diagonal matrices and $\Phi$ be the corresponding root system of type $\mr{A}_1$. Let $\epsilon_1, \epsilon_2 \in X^\ast(\T)$ be the usual coordinate functions on $\T$. As usual, we will use the additive notation  for the abelian group $X^\ast(\T)$ of characters of $\T$. The root system is given by $\Phi=\Phi^{+} \cup \Phi^{-}$, where $\Phi^{+}$ and $\Phi^{-}$ denote the set of positive and negative roots of $GL_2$ respectively, and $\Phi= \left\{\epsilon_1-\epsilon_2, \epsilon_1+\epsilon_2 \right\}$. Then the system of simple roots is defined by $\Delta=\left\{\alpha_1=\epsilon_1-\epsilon_2\right\}$. The fundamental weights associated to this root system are given by $\gamma_1 = \epsilon_1$ and $\gamma_2=\epsilon_1+\epsilon_2$. Note that in case of $SL_2$, $\gamma_2=\epsilon_1+\epsilon_2=0$. The irreducible finite dimensional representations of $GL_2$ are determined by their highest weight which in this case are the elements of the form $\lambda = m_1 \gamma_1 + m_2 \gamma_2$ with $m_1, m_2$ non-negative integers. The Weyl group $\mathcal{W}$ of $\Phi$ is given by the symmetric group  $\mathfrak{S}_2=\left\{ e=(1\,2), s:=(2 \, 1) \right\}$.

The above defined root system determines a proper standard $\mathbb{Q}$-parabolic subgroup $\mr{P}_{0}$-the Borel subgroup, \ie for every $\Q$-algebra $A$, 
\[ 
\mr{P}_{0}(A) = \left\{ \left( \begin{array}{cccccc}
\ast & \ast  \\
0 & \ast  \\
\end{array}  \right) \in GL_2(A) \right\} \,, 
\]

Consider the maximal connected compact subgroup $\mr{K}_\infty = \mr{O}(2) \subset GL_2(\mathbb{R})$ . Then
$GL_2(\mathbb{R})/\mr{K}_\infty,$ which is very well known object namely the upper half plane $\h:=\left\{ z = x+i y\in \C | y >0 \right\}$.
Let  $\Gamma = GL_2(\mathbb{Z})$, then $\mr{S}_\Gamma$
We denote by
$\mr{S}_\Gamma
=
\Gamma\backslash GL_2(\mathbb{R})/\mr{K}_\infty$
the locally symmetric space.
Let $\mc{W}=\{ e=(1\, 2), s:=(2 \, 1)$ be the set of Weyl representatives of the parabolic $\rP_{0}$. If $\rho \in X^\ast(\mr{T})$ denotes half of the sum of the positive roots (in this case this is just $\epsilon_1$) and $w \in \mathcal{W}$, then the element $w\cdot\lambda = w(\lambda + \rho) - \rho$ is a highest weight of an irreducible representation $\mathcal{M}_{w\cdot\lambda}$ of the Levi subgroup $\mr{M}$ of $\rP_{0}$ and defines a sheaf $\widetilde{\m}_{w\cdot\lambda}$ over $\rS_\Gamma^\mr{M}$. 

In the next table we list all the elements of the Weyl group along with their lengths and the element $w\cdot \lambda$, where $\lambda=m_1\lambda_1+\m_2\lambda_2$  is the coefficient of $\lambda_1$.  
\begin{table}[htbp]\label{weight-sl2}
\label{table0}
\centering
\begin{tabular}{clccccc}
\hline 
\noindent Label & $w$ & length of $w$ & $w\cdot \lambda$  \\
\hline
$w_1$ 	&	$e=(1\, 2)$	&	$0$	& $(m_1,m_2)$  \\ 
$w_2$	& 	$s=(2\, 1)$	& 	$1$	&  $(m_2-1,m_1+1)$ \\ 
\hline
\end{tabular}
\vspace{0.4cm}
\caption{The  Weyl group of $GL_2$}
\end{table}

\subsection{Cohomology of $SL_2(\Z)$ and $GL_2(\Z)$} The cohomology of these two arithmetic groups have been vastly studied. We summarize the result from~\cite{BHHM} discussed in Section 5 and take the opportunity to explain the approach we are taking in more general situation than rank 1.

\subsubsection{Boundary cohomology of $SL_2(\Z)$ and $GL_2(\Z)$} 
We consider cohomology of the boundary of the Borel-Serre compactfication of $X=GL_2(\R)/\R_{>0}\times O_2(\R)$. 
Intuitively, one attaches a circle in place of a cusp point

The cohomology of boundary (called boundary cohomology) of $GL_2(\Z)$ can be computed by 
the Kostant's Theorem \cite{K}.
Then \[H^0_{\partial}(GL_2(\Z),V_{m_1,m_2})=H^0(GL_1(\Z),V_{m_1})\otimes H^0(GL_1(\Z),V_{m_2}),\]
and
 \[H^1_\partial(GL_2(\Z),V_{m_1,m_2})=H^0(GL_1(\Z),V_{m_2-1})\otimes H^0(GL_1(\Z),V_{m_1+1}),\]
where $V_{m_1,m_2}$ is the highest weight representation of $GL_2(\Z)$ with weight $\lambda=m_1\lambda_1+\m_2\lambda_2$ and
$V_{m_1}=sign^{m_1}$ be a power of the sign representation of $GL_1(\Z)$.
Note that $H^0(GL_1(\Z),V_{m_1})=\C$ if $m_1$ is even and $H^0(GL_1(\Z),V_{m_1})=\C$ if $m_1$ is odd.
Then, we obtain the following.
\begin{prop}
(a) If both $m_1$ and $m_2$ are even then
 \[H^q_{\partial}(GL_2(\Z),V_{m_1,m_2})
 =
 \left\{
\begin{tabular}{ll}
$\C$,	&	for $q=0$\\
$0$,		&	for $q=1$.
\end{tabular}
 \right.
 \]
 (b) If both $m_1$ and $m_2$ are odd then
 \[H^q_{\partial}(GL_2(\Z),V_{m_1,m_2})
 =
 \left\{
\begin{tabular}{ll}
$0$,	&	for $q=0$\\
$\C$,		&	for $q=1$.
\end{tabular}
 \right.
 \]
(c) If $m_1$ and $m_2$ have different parity then
 \[H^0_{\partial}(GL_2(\Z),V_{m_1,m_2})=H^1_{\partial}(GL_2(\Z),V_{m_1,m_2})=0.\]
\end{prop}

We can use the homological Euler characteristics of $GL_2(\Z)$ and $SL_2(\Z)$ is order to compute the first cohomology. Note that $H^0$ is zero for when the irreducible representation different from the trivial one
We present the following isomorphism for intuition. One can recover a simple proof by using the data of the Table~\ref{eulersl2gl2} and the Kostant formula \cite{K}. 
Denote by $Sym^n$ the $n$-th symmetric power of the standard representation of $GL_2$. In terms of highest weights it can be written as $V_{n,0}$.
Also $Sym^n\otimes det$ is the same as $V_{n+1,1}$. We have

\[H^{1}(SL_2(\Z), Sym^n) =  H^{1}_{Eis}(SL_2(\Z), Sym^n) \oplus H^{1}_{cusp}(SL_2(\Z), Sym^n).\]
Also
\[H^{1}_{Eis}(SL_2(\Z), Sym^n)=H^{1}_{Eis}(GL_2(\Z), Sym^n)+H^{1}_{Eis}(GL_2(\Z), Sym^n\otimes det)\]
and
\begin{eqnarray*}
&&H^{1}_{cusp}(SL_2(\Z), Sym^n)=\\
&&= H^{1}_{cusp}(GL_2(\Z), Sym^n \otimes det) \oplus H^{1}_{cusp}(GL_2(\Z), Sym^n)=\\
&&=H^{1}_{cusp}(GL_2(\Z), Sym^n) \oplus H^{1}_{cusp}(GL_2(\Z), Sym^n).
\end{eqnarray*}

\section{Euler Characteristics and Cohomology of $\G_1(2,p)$}
Recall that $\G_1(2,p)$ is the subgroup of $GL_2(\Z)$ that fixes $(0,1) \mod p$.
Let $A$ be a torsion element of $\G_1(2,p)$
Since $p$ is relatively prime to $2$ and $3$, we have obtain that $t^2+1$ and $t^2\pm t+1$ cannot be characteristic polynomials of $A$ modulo $p$. Therefore, the eigenvalues of $A$ are $1,1$ or $1,-1$.

If the eigenvalues are $1,1$, then the matrix is the identity. Therefore, its centralizer would be $\G_1(2,p)$, whose Euler characteristic is $-\frac{1}{24}\varphi_2(p)$.
Note that $-\frac{1}{24}$ is the Euler characteristic of $GL_2(\Z)$ and the index of $\G_1(2,p)$ is $\varphi_2(p)=p^2-1$.

{\bf Definition:} (of $N_{G}^{\G}(A)$) Let $C_G(A)=\{X\in G : XAX^{-1}=A \}$  be the  centralizer of $A$ as an element of $A \in G$  and $N_{G}^{\G}(A)=\{X\in G : XAX^{-1} \in \G\}$ which will be referred as  the normalizer of $A$ in $G$ relative to $\G$.

\begin{prop}\label{G1 classes}
Let $G=GL_2(\Z)$, $\Gamma:=\Gamma_1(2,p)$ and $A$ be a torsion element of $\Gamma_1(2,p)$. Then the following is true.
\begin{enumerate}
\item[(a)] If the root $+1$ has multiplicity $1$ then the set $\Gamma\backslash N_G^\Gamma/C_G(A),$
has $\frac{1}{2}\varphi(p)$ elements.

\item[(b)] If the root $+1$ has multiplicity $2$ then the set $\Gamma\backslash N_G^\G(A)/C_G(A),$
has $1$ element.
\end{enumerate}
\end{prop}
\begin{proof}
(a) If the multiplicity of the eigenvalue $1$ is one, then we have that the other eigenvalue is $-1$. Let $X\in N_G^\Gamma$ and let $B=XAX^{-1}$. Then $B\in \Gamma$. We can write the matrices $A$, $B$ and $X$ in terms of elements. Then the relation between them is
\[
\left(
\begin{tabular}{cc}
  $x_{11}$   & $x_{12}$ \\
$x_{21}$     & $x_{22}$
\end{tabular}
\right)
\cdot
\left(
\begin{tabular}{cc}
  $a_{11}$   & $a_{12}$ \\
$0$     & $1$
\end{tabular}
\right)
=
\left(
\begin{tabular}{cc}
  $b_{11}$   & $b_{12}$ \\
$0$     & $1$
\end{tabular}
\right)
\cdot
\left(
\begin{tabular}{cc}
  $x_{11}$   & $x_{12}$ \\
$x_{21}$     & $x_{22}$
\end{tabular}
\right)
(\mr{mod} \,\, p)
\]
Consider the $(2,1)$-entry of the product. 
The left-hand side is $x_{21}a_{11}$ and 
the right-hand side is $1\cdot x_{21}$
Examine the map
\[P_{b_{22}a_{11}}(x_{21})=x_{21}a_{11} - b_{22} x_{21}\]
Since we have multiplicity $1$ of the eigenvalue $1$, we obtain that $a_{11}=-1\mod p$.
Then the map $P_{b_{22}a_{11}}$ is non-singular mod p. 
(In fact, it is $P_{b_{22}a_{11}}(x_{21})=-2x_{21}.$)
Therefore $P_{b_{22}a_{11}}(x_{21})=0 \mod p$ implies
$x_{21}=0\mod p$.
We obtain that
$X\in \Gamma_0(2,p).$
Recall that $X$ was any element of $N_G^\Gamma$. 
Therefore,
$N_G^\Gamma\subset\Gamma_0(2,p)$.
On the other hand $\Gamma_0(2,p)$ lies inside the normalizer of $\Gamma_1(2,p)$ inside $GL_2(\Z)$. From the definition of $N_G^\Gamma$ it follows that $\Gamma_0(2,p)\subset N_{G}^{\Gamma}$.
Therefore
$N_G^\Gamma=\Gamma_0(2,p)$.
Then the double quotient is 
$\Gamma\backslash N_G^\Gamma/C_{GL_2(\Z)}(A)=\Gamma_1(2,p)\backslash \Gamma_0(2,p)/C$,
where $C=\{I,-I,A,-A\}$ is the center of $A$ in $GL_2(\Z)$.
Also, $A$ is in $\Gamma_1(2,p)$, (however $-I\notin \Gamma_1(2,p)$).
Using that the index of $\Gamma_1(2,p)$ in $\Gamma_0(2,p)$ is $\varphi(p)=p-1$, we obtain that
$\Gamma\backslash N_G^\Gamma/C_{GL_2(\Z)}(A)$ has $\frac{1}{2}\varphi(p)$ elements.
that proves part (a).

For part (b), if $A$ has multiplicity two of the eigenvalue $1$, then there is only one such choice of $A$, namely, $A=I$.
Then $C_{GL_2(\Z)}(A)=GL_2(\Z)$. Also $N_G^\Gamma$ contains the centralizer therefore and it is contained in $GL_2(\Z)$. We obtain that $N_G^\Gamma=GL_2(\Z)$. It implies that the double quotient consists of one element.
\end{proof}

\begin{thm}
\label{chi2}
\[\chi_h(\Gamma_1(2,p),V)=\frac{1}{2}\varphi(p)Tr([1,-1]|V)-\frac{1}{24}\varphi_2(p)Tr(I_2|V).\]
\end{thm}
\proof
In order to compute the homological Euler characteristics of $\Gamma_1(2,p)$ with various coefficient systems, we have to examine it torsion elements.
In $GL_2(\Z)$ there are two non-conjugate torsion matrices with eigenvalues $1, -1$. 
Both conjugacy classes have representatives in $\Gamma_1(2,p)$, namely,
\[
T_1=\left[
\begin{tabular}{cc}
  $-1$   & $0$ \\
$0$     & $1$
\end{tabular}
\right]
\mbox{ and }
T_2=\left[
\begin{tabular}{cc}
  $-1$   & $1$ \\
$0$     & $1$
\end{tabular}
\right].
\]
From Proposition \ref{G1 classes}, we obtain that there are exactly  $\frac{1}{2}\varphi(p)$ different conjugacy classes in $\Gamma_1(2,p)$ whose element are conjugate to $T_1$. (Similarly for $T_2$.) For any element $A$ in any of the conjugacy classes, we have that its centralizer in $\Gamma_1(2,p)$ has two elements $I$ and $A$. (In $GL_2(\Z)$, it has four elements $\pm I$ and $\pm A$.

Thus, there are $\frac{1}{2}\varphi(p)$ conjugacy classes sitting above the conjugacy class of $T_1$. Similarly, for $T_2$. For both cases, we have total of $\varphi(p)$ conjugacy classes sitting $\Gamma_1(2,p)$ above them.
If we sum over the conjugacy classes in $\Gamma_1(2,p)$ sitting above the conjugacy classes of $T_1$ and $T_2$, we obtain
$\sum_A C_{\Gamma_1(2,p)}(A)=\frac{1}{2}\varphi(p)$.

If the torsion element is the identity then $\chi(C_{\Gamma_1(2,p)}(I_2))=\chi(\Gamma_1(2,p))=\varphi_2(p)\chi(GL_2(\Z))=-\frac{1}{24}\varphi_2(p),$
where $\varphi_2(p)=p^2-1$ is the index of $\Gamma_1(2,p)$ inside $GL_2(\Z)$. 

Therefore, Euler characteristics of $\Gamma_1(2,p)$ can be computed by
\[\chi_h(\Gamma_1(2,p),V)=\frac{1}{2}\varphi(p)Tr([1,-1]|V)-\frac{1}{24}\varphi_2(p)Tr(I_2|V)\]
\qed

\begin{lemma}
\label{tr GL2}
Let us denote by $S^n$ the $n$-th symmetric power of the standard representation. Then
\[
\begin{tabular}{cccc}
$Tr([1,-1]|S^n)
=\frac{1}{2}(1+(-1)^n)$,
& 
$Tr([1,-1]|S^n\otimes det)
=-\frac{1}{2}(1-(-1)^n),$\\
\\
$Tr([I_2]|S^n)
=n+1$,
& 
$Tr(I_2|S^n\otimes det)
=n+1$.
\end{tabular}
\]
\end{lemma}

\begin{cor}
\[
\chi_h(\Gamma_1(2,p),S^n)
=\frac{1}{4}(1+(-1)^n)\varphi(p)
-\frac{n+1}{24}\varphi_2(p),
\]
\[
\chi_h(\Gamma_1(2,p),S^n\otimes det)
=-\frac{1}{4}(1+(-1)^n)\varphi(p)
-\frac{n+1}{24}\varphi_2(p).
\]
\end{cor}
\proof
It follows directly from Theorem \ref{chi2} and Lemma \ref{tr GL2}.
\qed

\begin{lemma}The dimensions boundary cohomology of $\G_1(2,p)$, when it is non-trivial, is given below.

(a) $\dim H^0_\partial(\Gamma_1(2,p),\Q)=\varphi(p)$

(b) $\dim H^1_\partial(\Gamma_1(2,p),det)
=\varphi(p)$

(c) For even $n$ and $n>0$, we have
$
\dim H^0_\partial(\Gamma_1(2,p),S^n)=\varphi(p)
$

(d) For even $n$ and $n>0$, we have
$\dim H^1_\partial(\Gamma_1(2,p),S^n\otimes det)=\varphi(p)
$

(e) For odd $n$, we have
$
\dim H^0_\partial(\Gamma_1(2,p),S^n)=\frac{1}{2}\varphi(p)
$
and

$
\dim H^1_\partial(\Gamma_1(2,p),S^n)=\frac{1}{2}\varphi(p)
$

(g) For odd $n$, we have
$
\dim H^0_\partial(\Gamma_1(2,p),
S^n \otimes det)=\frac{1}{2}\varphi(p)
$
and

$
\dim H^1_\partial(\Gamma_1(2,p),
S^n \otimes det)=\frac{1}{2}\varphi(p)
$
\end{lemma}
\proof
In order to compute the boundary cohomology of $\G_1(2,p)$, we have to apply Kostant formula.
For part (a), we use that the trivial representation $Q$ has weight $(0,0)$. We are going to use that there are $\varphi(p)$ boundary components of $\G_1(2,p)$.
We obtain
$H^0_\partial(\G_1,(2,p),(0,0))=\varphi(p)(0|0)$
Then
$$\dim H^0_\partial(\G_1,(2,p),(0,0))=\varphi(p)$$
And 
$H^1_\partial(\G_1,(2,p),(0,0))=\varphi(p)(-1|1)=0$.
For part (b), we use that the representation $det$ has weight (1,1). We obtain
$H^0_\partial(\G_1(2,p),(1,1))=
\varphi(p)(1|1)=0$
 and
 $H^1_\partial(\G_1(2,p),(1,1))=
\varphi(p)(0|2)$.
Then
$\dim H^1_\partial(\G_1(2,p),(1,1))=
\varphi(p).$
For Part (c), we use that $S^{2n}$ has weight $(2n,0)$
We obtain
$H^0_\partial(\G_1,(2,p),(2n,0))=\varphi(p)(2n|0)$
Then
$$\dim H^0_\partial(\G_1,(2,p),(2n,0))=\varphi(p)$$
And 
$H^1_\partial(\G_1,(2,p),(2n,0))=\varphi(p)(-1|2n+1)=0$.
For part (d), we use that the representation $S^{2n}\otimes det$ has weight $(2n+1,1)$. We obtain
$H^0_\partial(\G_1(2,p),(2n+1,1))=
\varphi(p)(2n+1|1)=0$
 and
 $H^1_\partial(\G_1(2,p),(2n+1,1))=
\varphi(p)(0|2n+2)$.
Then
$\dim H^1_\partial(\G_1(2,p),(2n+1,1))=
\varphi(p).$
For part (e), we use that the representation $S^{2n+1}$ has weight $(2n+1,0)$.
 We obtain
$H^0_\partial(\G_1(2,p),(2n+1,0))=
\frac{1}{2}\varphi(p)(2n+1|0)$
 and
 $H^1_\partial(\G_1(2,p),(2n+1,0))=
\frac{1}{2}\varphi(p)(-1|2n+2)$.
Then
$\dim H^0_\partial(\G_1(2,p),(2n+1,0))=
\frac{1}{2}\varphi(p).$
$\dim H^1_\partial(\G_1(2,p),(2n+1,0))=
\frac{1}{2}\varphi(p).$
For part (f), we use that the representation $S^{2n+1}\otimes det$ has weight $(2n+1,0)$.
 We obtain
$H^0_\partial(\G_1(2,p),(2n+2,1))=
\frac{1}{2}\varphi(p)(2n+2|1)$
 and
 $H^1_\partial(\G_1(2,p),(2n+2,1))=
\frac{1}{2}\varphi(p)(0|2n+3)$.
Then
$\dim H^0_\partial(\G_1(2,p),(2n+2,1))=
\frac{1}{2}\varphi(p).$
$\dim H^1_\partial(\G_1(2,p),(2n+2,1))=
\frac{1}{2}\varphi(p).$
\qed
\begin{thm} 
The dimensions Eisenstein cohomology of $\G_1(2,p)$, is given below.

(a) $\dim H^0_{Eis}(\Gamma_1(2,p),\Q)=1$

(b) $\dim H^1_{Eis}(\Gamma_1(2,p),det)
=\varphi(p)-1$

(c) For even $n$ and $n>0$, we have
$
\dim H^1_{Eis}(\Gamma_1(2,p),S^n)=0
$

(d) For even $n$ and $n>0$, we have
$\dim H^1_{Eis}(\Gamma_1(2,p),S^n\otimes det)=\varphi(p)
$

(e) For odd $n$, we have
$
\dim H^1_{Eis}(\Gamma_1(2,p),S^n)=\frac{1}{2}\varphi(p)
$

(g) For odd $n$, we have
$
\dim H^1_{Eis}(\Gamma_1(2,p),
S^n \otimes det)=\frac{1}{2}\varphi(p)
$
\end{thm}
\proof 
We have that 
\[H^i_{Eis}(\G_1(2,p)\cap SL_2(\Z),V)=
H^i_{Eis}(\G_1(2,p),V)
H^i_{Eis}(\G_1(2,p),V\otimes det)
\]
and
\begin{eqnarray*}
2\varphi(p)
=&&\dim H^0_{\partial}(\G_1(2,p)\cap SL_2(\Z),V)+\\
&&+\dim H^1-\partial(\G_1(2,p)\cap SL_2(\Z),V)\\
=&&2(\dim H^0_{Eis}(\G_1(2,p)\cap SL_2(\Z),V)+\\
&&+\dim H^1-{Eis}(\G_1(2,p)\cap SL_2(\Z),V))
\end{eqnarray*}
We have that $H^0(\G_1(2,p)\cap SL_2(\Z),\Q)=\Q$. Therefore, $H^0_{Eis}(\G_1(2,p),Q)=H^0(\G_1(2,p),Q)=Q$.
We obtain that 
$\dim H^1_{Eis}(\G_1(2,p),det)=\varphi(p)-1$.
We for the representation $S^{2n}$, we have
$H^0(\G_1(2,p),S^{2n})=0$.
Therefore, its Eisenstein cohomology vanishes. We obtain that $H^1_{Eis}(\G_1(2,p),S^{2n}\otimes det)=\varphi(p) $
For $S^{2n+1}$, we have
$H^0_{Eis}(\G_1(2,p),S^{2n+1})=0$
and
$H^0_{Eis}(\G_1(2,p),S^{2n+1}
\otimes)=0$
Also, in the Poincare duality for the cuspidal cohomology we obtain that 
$H^1_{cusp}(\G_1(2,p),S^{2n+1},S^{2n+1}\otimes det \otimes \C)$
is dual to 
$H^1_{cusp}(\G_1(2,p),S^{2n+1} \otimes \C)$.
Also, from the equality between the Euler characteristics 
\[\chi_h(\G_1(2,p),S^{2n+1}
\otimes det)
=\chi_h(\G_1(2,p),S^{2n+1}),\]
we obtain that 
\[\dim H^1_{Eis}(\G_1(2,p),S^{2n+1}
\otimes det)
=\dim H^1_{Eis}(\G_1(2,p),S^{2n+1}).\]
Also 
\[\dim H^1_{Eis}(\G_1(2,p),S^{2n+1}
\otimes det)
+\dim H^1_{Eis}(\G_1(2,p),S^{2n+1})=\varphi(p).\]
Therefore
\[\dim H^1_{Eis}(\G_1(2,p),S^{2n+1}
\otimes det)
=\dim H^1_{Eis}(\G_1(2,p),S^{2n+1})=\frac{1}{2}\varphi(p).\]

\begin{cor}
\label{Gamma1(2,p) cusp}
\[\dim H^1_{cusp}(\Gamma_1(2,p),\C)=1+\frac{1}{24}\varphi_2(p) -\frac{1}{2}\varphi(p)
\]
\[\dim H^1_{cusp}(\Gamma_1(2,p),det)
=1+\frac{1}{24}\varphi_2(p) -\frac{1}{2}\varphi(p).
\]
For $n>0$
\[\dim H^1_{cusp}(\Gamma_1(2,p),S^n)=\frac{n+1}{24}\varphi_2(p)-\frac{1}{2}\varphi(p).
\]
\[\dim H^1_{cusp}(\Gamma_1(2,p),S^n\otimes det)=\frac{n+1}{24}\varphi_2(p)
-\frac{1}{2}\varphi(p).\]
\end{cor}
\proof
For the trivial representation we have
\[\chi_h(\Gamma_1(2,p),\Q)
=
\dim H^0(\Gamma_1(2,p),\Q) 
-\dim H^1(\Gamma_1(2,p),\Q)\]
The Eisenstein cohomology occurs only in dimension zero. It is one dimensional.
Therefore, the first cohomology coincides with the cusp cohomology.
$
\dim H^1_{cusp}(\Gamma_1(2,p),\Q)
=
\dim H^1(\Gamma_1(2,p),\Q)
=$
Using Corollary 13, we obtain that
\[
\dim H^1_{cusp}(\Gamma_1(2,p),\Q)
=
\dim H^0(\Gamma_1(2,p)
-\chi_h(\Gamma_1(2,p),\Q)
=
1-\frac{1}{2}\varphi+\frac{1}{24}\varphi_2(p).
\]
For all other irreducible representations, we have 
$H^0(\Gamma_1(2,p),V)=0$.
Therefore,
$\dim H^1(\Gamma_1(2,p),V)=-\chi_h(\Gamma_1(2,p),V)$.
If $V=det$ then from Theorem 14 (b) we have $H^1_{Eis}(\Gamma_1(2,p),det)=-1+\varphi(p).$
Therefore
\begin{eqnarray*}
\dim H^1_{cusp}(\Gamma_1(2,p),det)
&&=
-\dim H^1_{Eis}(\Gamma_1(2,p),det)
-\chi_h(\Gamma_1(2,p),det)\\
&&=-(-1+\varphi(p))
-\left(-\frac{1}{2}\varphi(p)
-\frac{1}{24}\varphi_2)p)
\right)\\
&&=1-\frac{1}{2}\varphi(p)
+\frac{1}{24}\varphi_2(p).
\end{eqnarray*}
If $n$ is a positive even integer then
\begin{eqnarray*}
\dim H^1_{cusp}(\Gamma_1(2,p),S^n)
&&=
-\dim H^1_{Eis}(\Gamma_1(2,p),S^n)
-\chi_h(\Gamma_1(2,p),S^n)\\
&&=-0
-\left(\frac{1}{2}\varphi(p)
-\frac{n+1}{24}\varphi_2(p)
\right)\\
&&=-\frac{1}{2}\varphi(p)
+\frac{n+1}{24}\varphi_2(p).
\end{eqnarray*}
Also,
\begin{eqnarray*}
\dim H^1_{cusp}(\Gamma_1(2,p),S^n\otimes det)
&&=
-\dim H^1_{Eis}(\Gamma_1(2,p),S^n\otimes det)
-\chi_h(\Gamma_1(2,p),S^n\otimes det)\\
&&=-\varphi(p)
-\left(-\frac{1}{2}\varphi(p)
-\frac{n+1}{24}\varphi_2(p)
\right)\\
&&=-\frac{1}{2}\varphi(p)
+\frac{n+1}{24}\varphi_2(p).
\end{eqnarray*}

If $n$ is a positive odd integer then
\begin{eqnarray*}
\dim H^1_{cusp}(\Gamma_1(2,p),S^n)
&&=
-\dim H^1_{Eis}(\Gamma_1(2,p),S^n)
-\chi_h(\Gamma_1(2,p),S^n)\\
&&=-\frac{1}{2}\varphi(p)
-
\left(
-\frac{n+1}{24}\varphi_2(p)
\right)\\
&&=-\frac{1}{2}\varphi(p)
+\frac{n+1}{24}\varphi_2(p).
\end{eqnarray*}
Also,
\begin{eqnarray*}
\dim H^1_{cusp}(\Gamma_1(2,p),S^n\otimes det)
&&=
-\dim H^1_{Eis}(\Gamma_1(2,p),S^n\otimes det)
-\chi_h(\Gamma_1(2,p),S^n\otimes det)\\
&&=-\frac{1}{2}\varphi(p)
-\left(
-\frac{n+1}{24}\varphi_2(p)
\right)\\
&&=-\frac{1}{2}\varphi(p)
+\frac{n+1}{24}\varphi_2(p).
\end{eqnarray*}
\qed

\section{General Setting for Euler Characteristics of $\Gamma_1(m,p)$}
We are going to compute the homological Euler characteristics of $\Gamma_1(m,p)$ with coefficients in various representations. We recall that $\Gamma_1(m,p)$ is the congruence subgroup of $GL_m(\Z)$ that fixes the vector $(0,\dots,0,1)$ modulo a prime $p$. For most of the results work for primes $p\geq 5$. 

We are going to use the formula 
\begin{equation}
\label{EulerChar}
\chi_h(G,V)=\sum_{(A)}\chi(C(A))Tr(A^{-1}|V)
\end{equation}
where the sum is taken over all conjugacy classes of torsion elements in the group $G$ and where $V$ is a representation of $G$.

From now on, We will use the notation $\Gamma=\Gamma_1(m,p)$ and $G=GL_m(\Z)$

In order to make use of the above formula, Equation \ref{EulerChar}, we have to group together torsion elements of $\Gamma$ that become conjugate to each other inside $G$. We will define a set $N^{\Gamma}_G$ with the property that the double quotient $\Gamma\backslash N_G^{\Gamma}/C_G(A)$ parametrizes all non-conjugate classes of elements in $\Gamma$ that become conjugate inside $G$, where $C_G(A)$ is the centralizer of the element $A$ inside the group $G$. The notation $\Gamma\backslash N_G^{\Gamma}/C_G(A)$ is used in a paper by Kenneth Brown \cite{B2}.

\begin{lemma}
Let $G$ be a group and $\Gamma$ be a subgroup of $G$. Given an element $A\in \Gamma$ the set of conjugacy classes of elements which become conjugate to $A$ inside $\Gamma$ is parametrized by the double quotient $\Gamma\backslash N_G^{\Gamma}/C_G(A)$,
where 
\[N_G^{\Gamma}=\{X\in G: XAX^{-1}\in \Gamma\}\]
and $C_G(A)$ is the centralizer of $A$ inside the group $G$.
\end{lemma}
\proof Let $A_1$ and $A_2$ be two elements of $\Gamma$ both conjugate to $A$ in $G$ and conjugate to each other inside $\Gamma$. Then, there exist elements of $X_1$ and $X_2$ in the bigger group $G$ such that 
$A_1=X_1AX_1^{-1}$ and $A_2=X_2AX_2^{-1}$.
Since $A_1$ and $A_2$ are conjugate to each other inside $\Gamma$, there exists $Y\in\Gamma$ such that 
$A_2=YA_1Y^{-1}$.
Then,
\[YX_1AX_1^{-1}Y^{-1}=X_2AX_2^{-1}.\]
Equivalently,
\[(X_2^{-1}YX_1)A=A(X_2^{-1}YX_1).\]
We obtain that $X_2^{-1}YX_1$ is in the centralizer of the element $A$ inside $G$. That is, $C_G(A)$. Therefore,
$YX_1\in X_2\cdot C_G(A)$ and \[X_1\in \Gamma\cdot C_G(A).\]
Therefore, $X_1$ and $X_2$ belong to the same double quotient $\Gamma\backslash N_G^{\Gamma}/C_G(A)$.

Conversely, suppose $X_1$ and $X_2$ belong to the same double quotient. We are going to show that both elements $X_1AX_1^{-1}$ and $X_2AX_2^{-1}$ belong to $\Gamma$ and that they are conjugate to each other inside $\Gamma$.

By definition of $N_\G^{\Gamma}$, we have that $X_1AX_1^{-1}$ and $X_2AX_2^{-1}$ belong to $\Gamma$. Since $X_1$ and $X_2$ belong to the same double quotient, we have that $X_2=YX_1C$, where $Y\in\Gamma$ and $C\in C_G(A)$. Then 
\begin{eqnarray*}
X_2AX_2^{-1}	&&=(YX_1C)A(YX_1X)^{-1}\\
			&&=YX_1CAC^{-1}X_1^{-1}Y^{-1}\\
			&&=YX_1AX_1^{-1}Y^{-1}\\
			&&=Y(X_1AX_1^{-1})Y^{-1}.\\
\end{eqnarray*}
Thus, $X_2AX_2^{-1}$ and $X_1AX_1^{-1}$ are conjugate to each other inside $\Gamma$.
\qed

\begin{lemma}

Let $A$ be a $k$-torsion element of $\G_1(m,p)$ for $k=1,2,3,4,6$, and let $p\geq 5$ be a prime. Then $1$ is an eigenvalue of $A$.
\end{lemma}
\begin{proof} Let $L$ be the field obtained by adjoining all eigenvalues of $A$ to $\Q$. Let $\mathfrak{p}$ be a prime ideal siting above $p$. Denote, also by $\mu(L)$ the roots of unity in $L$. Consider the projection $\pi:L\rightarrow L/\mathfrak{p}=\mathbb{F}_q$. Then $\pi$ maps $\mu(L)$ onto $\mathbb{F}_q$, because there are no $p$-roots of $1$ in $L$, since $p$ does not divide $k$. However, $1$ is an eigenvalue of $A$ mod $\mathfrak{p}$. Therefore, $1$ is an eigenvalue of $A$ in $L$.\end{proof}

\begin{lemma}
\label{tor type}
If $A$ is a torsion element of $GL_m(\Z)$ and if the Euler characteristic of its centralizer is not zero, then its eigenvalues are $1,-1,i,-i,w,w^2,w^4$ or $w^5$, where $w$ is primitive $6^{th}$-root of unity. The eigenvalues $1$ and $-1$ must have multiplicity at most two and the rest multiplicity one.
\end{lemma}
\proof
This is a property of the algebraic group of centralizer of $A$. Then, $C(A)$ is a product of algebraic groups of the type $GL_{n}$ over $\Q(\mu)$. If $n>2$, then or if $m>1$ and $\mu\neq \pm1$ then the corresponding arithmetic group will have vanishing Euler characteristic. Finally, if $m=1$ and $\Q(\mu)$ is of degree more that $2$, then the algebraic group is $\mathbb{G}_m$ gives an arithmetic group commensurable to $\Z[\mu]$. Since $\Q(\mu)$ has degree more that $2$, we obtain that $\Z[\mu]$ has a non-trivial free subgroup. Therefore, its Euler characteristic vanishes.\qed

\begin{cor}
Let $A\in \Gamma_1(m,p)$ be a torsion element with $\chi(C(A))\neq 0$. Let $p\geq 5$ be a prime number and let $f$ we the characteristic polynomial of $A$. Then $1$ is a root of $f$ and it has multiplicity $1$ or $2$.
\end{cor}
\begin{proof} It follows directly from the previous two lemmas.\end{proof}

\begin{lemma}
\label{N_G}
Let $A\in\Gamma$ be a torsion element with $\chi(C_G(A))\neq 0$. Let $p\geq 5$ be a prime number. Let $f$ be the characteristic polynomial of $A$. 
Then,  by Lemma \ref{tor type}, we have that $+1$ is an eigenvalue of $A$ and it has multiplicity $1$ or $2$.

(a) If the multiplicity of $+1$ is $1$, then
$\Gamma\backslash N_G^{\Gamma}/C_G(A)$,
has $\frac{1}{2}\phi(p)=\frac{1}{2}(p-1)$ elements.

(b) If the multiplicity of $+1$ is $2$, then
$\Gamma\backslash N_G^{\Gamma}/C_G(A)$,
has $\frac{1}{2}\phi(p)=\frac{1}{2}(p-1)$ has one element.
\end{lemma}
\proof
(a) Let $X\in N_G^\Gamma(A)$ and $B=XAX^{-1}$. Then $B\in\Gamma$. We can write the matrices $A$, $B$ and $X$ in a $2\times 2$ block-diagonal form with $A_{11}$, $B_{11}$ and $X_{11}$ of size $(m-1)\times(m-1)$ and 
$A_{22}$, $B_{22}$ and $X_{22}$ of size $1 \times 1)$.
Then
\[
\left[
\begin{matrix}
X_{11}	& X_{12}\\
X_{21}	& X_{22}
\end{matrix}
\right]
\cdot
\left[
\begin{matrix}
A_{11}	& A_{12}\\
0		& 1
\end{matrix}
\right]
=
\left[
\begin{matrix}
B_{11}	& B_{12}\\
0		& 1
\end{matrix}
\right]
\cdot
\left[
\begin{matrix}
X_{11}	& X_{12}\\
X_{21}	& X_{22}
\end{matrix}
\right]
\mod p.
\]
Consider the $(2,1)$-block of the product. The left hand side is $X_{21}A_{11} \mod p$ and the right hand side is $1\cdot X_{21} \mod p$.
Examine the map $P_{B_{22}A_{11}}:X_{21}\mapsto X_{21}A_{11}-B_{22}X_{21}$, with $B_{22}=1$. The eigenvalues of $A_{11}$ and $B_{22}$ are different.
Thus, the eigenvalues of $P_{B_{22}A_{11}}$ are different from zero. Therefore, $P_{B_{22}A_{11}}$ is non-singular $\mod p$. Then, 
$P_{B_{22}A_{11}}(X_{21})=0\mod p$ implies that $X_{21}=0 \mod p$.
Therefore, $N_G^\Gamma\subset \Gamma_0(m,p)$,
where  $\Gamma_0(m,p)$ is the subgroup of $GL_m(\Z)$ that sends the vector $(0,\cdot,0,1)$ to 
$(0,\cdot,0,a)$ modulo $p$ for any $a\neq 0 \mod p$.

On the other hand $\Gamma_0(m,p)$ lies inside the normalizer of $\Gamma_1(m,p)$. Also $\Gamma_0(m,p)\subset N_G^\Gamma$ from the definition of $N_G^\Gamma$.
Therefore, we obtain that $N_G^\Gamma=\Gamma_0(m,p)$.

Let $X$ be an element of $\Gamma_0(m,p)$ of the same block form as before. Inside the quotient $\Gamma_1(m,p)\backslash\Gamma_0(m,p)$, the element $X$is determined by $X_{22}\mod p$. Note that there are $\varphi(p)=p-1$ options for $X_{22}\neq 0 \mod p$.
The conditions on the centralizers of $A$ inside $\Gamma$ and inside $G$ are determined $\mod p$. We have that $G \mod p = SL_m^\pm(\F_p)$. Also $A_{21}=0 \mod p$.
Let $C$ be an element of the centralizer of $A$ inside $G$. Then $C_\Gamma(A)\backslash C_G(A)$ is determined by $C_{22}$ which could be $\pm 1$.  Therefore, the double quotient $\Gamma\backslash N_G^{\Gamma}/C_G(A)$ has $\frac{1}{2}(p-1)$ elements. 

(b) If $n=2$, then $A=I_2$, then $N_G^\Gamma=GL_2(\Z)$ and $C_G(A)=GL_2(\Z)$. Therefore the double quotient $\Gamma\backslash N_G^{\Gamma}/C_G(A)$ has one element.
Let $n>2$. By the Block -triangular Theorem (\cite{Thesis}, Theorem 1.3), we can conjugate $A$ by an element of $G$ to a matrix of the form
\[
\left[
\begin{matrix}
A_{11}	& A_{12}	& A_{13}\\
0		& 1		& 0\\
0		& 0		& 1
\end{matrix}
\right]
\]
Note that the block $A_{11}$ is a $(m-2)\times (m-2)$ matrix and that it does not have an eigenvalue $1$. We assume that $A$ is of the above form. Let $X$ be an element of $N_G^\Gamma$, and let $B=XAX^{-1}$. If we write this equation with respect to the block from of $A$, we obtain 
\begin{eqnarray*}
&&\left[
\begin{matrix}
X_{11}	& X_{12}	& X_{13}\\
X_{21}	& X_{22}	& X_{23}\\
X_{31}	& X_{32}	& X_{33}
\end{matrix}
\right]
\cdot
\left[
\begin{matrix}
A_{11}	& A_{12}	& A_{13}\\
0		& 1		& 0\\
0		& 0		& 1
\end{matrix}
\right]
=\\
&&=
\left[
\begin{matrix}
B_{11}	& B_{12}	& B_{13}\\
B_{21}	& B_{22}	& B_{23}\\
0		& 0		& 1
\end{matrix}
\right]
\cdot
\left[
\begin{matrix}
X_{11}	& X_{12}	& X_{13}\\
X_{21}	& X_{22}	& X_{23}\\
X_{31}	& X_{32}	& X_{33}
\end{matrix}
\right]
\mod p.
\end{eqnarray*}
Consider the $(3,1)$-block of the product. From the left hand side we obtain $X_{31}A_{11}$ and from the right hand side we have $1\cdot X_{31}$. Thus, 
$X_{31}A_{11}-1\cdot X_{31} = 0 \mod p$. We need to examine the map 
\[X_{31}\mapsto X_{31}A_{11}-B_{33}\cdot X_{31} = 0 \mod p\] with $B_{33}=1.$
Let \[P_{B_{33}A_{11}}(X_{31})= X_{31}\mapsto X_{31}A_{11}-B_{33}\cdot X_{31}.\] If $\lambda_k$ is an eigenvalue of $A_{11}$ then $\lambda_k-1$ is an eigenvalue of 
$P_{B_{33}A_{11}}(X_{31}).$ 
Moreover, all eigenvalues of $P_{B_{33}A_{11}}(X_{31})$ are of this form (Thm ???). 
Therefore the linear map $P_{B_{33}A_{11}}(X_{31})$ is non-singular. 
Since $A_{11}$ is a torsion matrix with eigenvalues $\{-1,\pm i,w^k\}$ ,where $w$ is a primitive $6$-root of $1$, we obtain that the determinant of 
$P_{B_{33}A_{11}}(X_{31})$ is can be divided only by powers of $2$ times powers of $3$.
Thus, $P_{B_{33}A_{11}}(X_{31})(X_{31})=0 \mod p$ implies that $X_{31} = 0 \mod p$. 

Let $N_{31}=(X_{ij})$ be the set of matrices in $GL_m(\Z)$ of the same block form with $X_{31}=0\mod p$. Then $N_G^\Gamma\subset N_{31}$.
We are going to show that $\mod p$ we have the following inclusions.
\[N_G^\Gamma\subset N_{31}\subset \Gamma_1(m,p)\cdot H\subset \Gamma_1(m,p)\cdot C_{GL_m(\Z)}(A)\subset N_G^\Gamma,\]
Where $H$ is a subgroup of $C_{GL_m(\Z)}(A)$. We just proved the first inclusion. The last inclusion holds by definition of $N_G^\Gamma$. It is the double coset of the identity. Modulo a prime $l\neq p$, we have that $\Gamma_1(m,p)=SL_m^\pm(\F_l)=N_G^\Gamma(A) \mod l,$
where $SL_m^\pm$ is the group of $m\times m$ matrices of determinant $\pm 1$. Therefore, the only restriction on $N_G^\Gamma(A)$ occurs only $\mod p$. The second to the last inclusion holds because $H$ is a subgroup of $C_{GL_m(\Z)}(A)$.

We are going to prove the second inclusion. Let 
\[\overline{C}_{12}=[C_{12}, C_{13}],\]
\[
\overline{C}_{22}
=
\left[
\begin{matrix}
C_{22}	& C_{23}\\
C_{23}	& C_{33}\\ 
\end{matrix}
\right]
\]
and let
\[
A
=
\left[
\begin{matrix}
A_{11}	& \overline{A}_{12}\\
0		& I_2\\ 
\end{matrix}
\right] 
\mod p
\]
Then $C$ being inside $C_G(A)$ implies that 
\[
\left[
\begin{matrix}
A_{11}	& \overline{A}_{12}\\
0		& I_2\\ 
\end{matrix}
\right] 
\cdot
\left[
\begin{matrix}
C_{11}	& \overline{C}_{12}\\
0		& \overline{C}_{22}\\ 
\end{matrix}
\right] 
=
\left[
\begin{matrix}
C_{11}	& \overline{C}_{12}\\
0		& \overline{C}_{22}\\ 
\end{matrix}
\right] 
\cdot
\left[
\begin{matrix}
A_{11}	& \overline{A}_{12}\\
0		& I_2\\ 
\end{matrix}
\right] 
\mod p
\]
Consider the $(1,2)$-blocks on both sides,
The left hand side gives
$C_{11}\overline{A}_{12}+\overline{C}_{12} \mod p$ and the right hand side gives
$A_{11}\overline{C}_{12}+\overline{A}_{12}\overline{C}_{22} \mod p$.
Assume that $A_{11}$ is the $(m-2)\times (m-2)$ identity matrix $I_{m-2}$ modulo $p$.
Therefore,
$C_{11}\overline{A}_{12}+\overline{C}_{12}
=
A_{11}\overline{C}_{12}+\overline{A}_{12}\overline{C}_{22} \mod p$.
Let us choose $C_{11}$ to be congruent to $I_{m-2}$ modulo $p$. 
Then 
$\overline{A}_{12}+\overline{C}_{12}
=
A_{11}\overline{C}_{12}+\overline{A}_{12}\overline{C}_{22} \mod p$.
And 
$(1-A_{11})\overline{C}_{12}
=
\overline{A}_{12}(\overline{C}_{22}-1) \mod p$.
If $A$ is fixed then for any $C_{22}$ there is a unique $\overline{C}_{12}$, since the eigenvalues of $A_{11}$ are different from $+1$.

We come back to proving that $N_{31}\subset \Gamma_1(m,p)\cdot H$,
where $H$ is a certain subgroup of $C_G(A)$.

Let $X\in \Gamma_1(m,p)$ and let $C\in H$.
Consider both of them in block diagonal form. We express
\[
X
=
\left[
\begin{matrix}
X_{11}	& X_{12}	& X_{13}\\
X_{21}	& X_{22}	& X_{23}\\
X_{31}	& X_{32}	& X_{33}
\end{matrix}
\right]
\mod p
\]
and
\[
C
=
\left[
\begin{matrix}
I_{m-2}	& C_{12}	& C_{13}\\
0		& C_{22}	& C_{23}\\
0		& C_{32}	& C_{33}
\end{matrix}
\right]
\mod p.
\]
in block form with $(1,1)$-block being an $(m-2)\times(m-2)$ matrix and $(2,2)$-block and $(3,3)$-block being $1\times 1$ matrices.
Let
\[
\overline{X}
=
\left[
\begin{matrix}
X_{12}	& X_{13}\\
X_{22}	& X_{23}
\end{matrix}
\right]
\mod p
\]
and
\[
\overline{C}
=
\left[
\begin{matrix}
C_{22}	& C_{23}\\
C_{32}	& C_{33}
\end{matrix}
\right]
\mod p.
\]
Then,
\[X\cdot C
=
\left[
\begin{matrix}
X_{11}	& *		& *\\
X_{21}	& *		& *\\
0		& C_{32}	& C_{33}
\end{matrix}
\right]
\mod p
\]
The vectors $(X_{11}, X_{21})$ and $(C_{32}, C_{33})$ can be any nonzero vectors modulo $p$. 
The matrix 
\[
\left[
\begin{matrix}
*		& *\\
*		& *
\end{matrix}
\right]
\]
is equal to
\[
\left[
\begin{matrix}
X_{11}\\
X_{21}
\end{matrix}
\right]
\cdot
[C_{32}, C_{33}]
+
\overline{X}\cdot\overline{C}
\mod p
\]

Claim: When the vectors $(X_{11}, X_{21})$ and $(C_{32}, C_{33})$ are fixed, the matrix \[
\left[
\begin{matrix}
X_{11}\\
X_{21}
\end{matrix}
\right]
\cdot
[C_{12}, C_{13}]
+
\overline{X}\cdot\overline{C}
\mod p
\]
can have any entries.

Indeed, let us fix $(C_{32}, C_{33})$. Choose $(C_{22}, C_{23})$ to be $(1,0)$ or $(0,1)$, so that the matrix $\overline{C}$ is invertible. Let $\overline{Z}$ be ant matrix with dimensions of $\overline{X}\cdot\overline{C}$. Let $\overline{X}=\overline{Z}\cdot\overline{C}^{-1}$. Then $\overline{X}\cdot\overline{C}=\overline{Z}\cdot\overline{C}^{-1}\overline{C}=\overline{Z}$ can be any matrix. Thus,
$\Gamma_1(m,p)\cdot H=N_G^\Gamma$.
In particular $N_G^\Gamma\subset\Gamma_1(m,p)\cdot H$.
\qed

\begin{thm}
\label{Gamma char}
Let $V$ be any finite dimensional highest weight representation of $GL_m$ and let $p\geq 5$ be any prime. 
Then the homological Euler characteristic of $\Gamma_1(m,p)$ is
\begin{eqnarray*}
\chi_h(\Gamma_1(m,p),V)
=
&&(p-1)\sum^{(1)}_AR(A)\chi(C_G(A))Tr(A|V)\\
&&+(p^2-1)\sum^{(2)}_{A}R(A)\chi(C_G(A))Tr(A|V),
\end{eqnarray*}
where the sum is over block diagonal matrices with blocks chosen from the set $\{+1,-1,+I_2,-I_2,T_3,T_4,T_6\}$ without repetition and  
$
T_3
=
\left[
\begin{matrix}
0	& 1\\
-1	& -1
\end{matrix}
\right]
\mbox{, }
T_4
=
\left[
\begin{matrix}
0	& 1\\
-1	& 0
\end{matrix}
\right]
\mbox{ and }
T_6
=
\left[
\begin{matrix}
0	& -1\\
1	& 1
\end{matrix}
\right].
$
Also, at most one of $+1$ or $I_2$ occurs and at most one of $-1$ or $-I_2$ occurs.
The first summation $)\sum^{(1)}_A$ is over those $A$ such that the block $+1$ occurs.
The first summation $)\sum^{(1)}_A$ is over those $A$ such that the block $+I_2$ occurs.
\end{thm}
\proof
Let $A$ be a torsion element of $\Gamma_1(m,p)$ such that $\chi(C_\Gamma(A))\neq 0$. Then $\chi(C_G(A))\neq 0$.
By Lemma \ref{tor type}, we have that $1$ is a negenvalue of $A$ and that the eigenvalue $+1$ has multiplicity one or two. 

We are going to show the following.
\begin{lemma}
\label{Gamma center}
(a) If the eigenvalues of a torsion matrix $A\in \Gamma$ has eigenvalue $+1$ with multiplicity one, then
\[\chi(C_\Gamma(A))=2\chi(C_G(A)).\]
(b) If the eigenvalues of a torsion matrix $A\in \Gamma$ has eigenvalue $+1$ with multiplicity two, then
\[\chi(C_\Gamma(A))=(p^2-1)\chi(C_G(A)).
\]
\end{lemma}
\proof (a) Let $B$ be a torsion element of $\Gamma=\Gamma_1(m,p)$ with eigenvalue $+1$ with multiplicity one. If we conjugate $B$ be a matrix $X\in GL_m(\Z)$, we can obtain
a matrix $A$ of the form
\[
A
=
\left[
\begin{matrix}
A_{11}	& A_{12}\\
0		& 1
\end{matrix}
\right],
\]
which is also an element of $\Gamma$. The centralizer of $A$ inside $\Gamma$ has index two inside the centralizer of $A$ inside $G=GL_m(\Z)$. 
If $C\in C_\Gamma(A)$ then $-C\in C_G(A)$. Conversely, if $CA=AC$, where $C$ is in $G$ then either $C$ or $-C$ is in $\Gamma$ because $C_{22}=\pm1$.
Also,
$C_G(B)=C_G(X^{-1}AX)=C_G(A)$. 
Now, let $\widetilde{C}$ or $-\widetilde{C}$ be in $C_\Gamma(B)$. Then $\widetilde{C}B=B\widetilde{C}$.
Conjugate both sides with $X$,then

$(X\widetilde{C}X^{-1})(XBX^{-1})=(XBX^{-1})(X\widetilde{C}X^{-1})$

$(X\widetilde{C}X^{-1})A=A(X\widetilde{C}X^{-1})$.

Therefore, $(X\widetilde{C}X^{-1})\in C_G(A)$ and $(X\widetilde{C}X^{-1})$ or $-(X\widetilde{C}X^{-1})$ is in $C_\Gamma(A)$, which is an exclusive or.
Thus, $C_\Gamma(A)$ and $C_\Gamma(B)$ are isomorphic. Therefore, $\chi(C_\Gamma(B))=\chi(C_\Gamma(A))=2\chi(C_G(A))$.
Also, there are $\frac{1}{2}(p-1)$ conjugacy classes of elements in $\Gamma$ which become conjugate to $A$ in $G$. 
If we sum over all conjugacy classes (B) in $\Gamma$ that become conjugate to (A) each other in $G$ then
\begin{eqnarray*}
\sum_{(B)}\chi(C_\Gamma(B))Tr(B^{-1}|V)
&& =\sum_{(B)}2\chi(C_G(B))Tr(B|V)\\
&& =2\frac{1}{2}(p-1)\chi(C_G(A))Tr(A|V)\\
&& =(p-1)\chi(C_G(A))Tr(A|V)
\end{eqnarray*}

(b) If $+1$ is an eigenvalue of $A$ with multiplicity two, then it is conjugate inside $G$ to a matrix of the form 
\[
\widetilde{A}
=
\left[
\begin{matrix}
A_{11}	& A_{12}\\
0		& I_2
\end{matrix}
\right],
\]
However by Lemma \ref{N_G} we have that the double quotient $\Gamma\backslash N_G^\Gamma/C_G(A)$ has one element. Therefore, $A$ and $\widetilde{A}$ are conjugate to each other inside $\Gamma$. 

Let $C$ be a matrix in the centralizer of $A$ in the same block diagonal form as $A$.
Then
\[
C
=
\left[
\begin{matrix}
C_{11}	& C_{12}\\
0		& C_{22}
\end{matrix}
\right],
\]
where $C_{11}$ is in the centralizer of $A_{11}$, $C_{22}$ is in the centralizer of $I_2$. For $C_{12}$ the conditions are a little bit more complicated.
From $AC=CA$ we can compare the $(1,2)$-block. Then $A_{11}C_{12}+A_{12}C_{22}=C_{11}A_{12}+C_{12}A_{22}$
Equivalently, $A_{11}C_{12}-C_{12}A_{22}=C_{11}A_{12}-A_{12}C_{22}$. 
Let $P_{A_{11}A_{22}}$ be the linear map $C_{12}\mapsto A_{11}C_{12}-C_{12}A_{22}$. From the Lemma \ref{P} $P_{A_{11}A_{22}}$ is non-singular. 
Then $C_{12}= P_{A_{11}A_{22}}^{-1}(C_{11}A_{12}-A_{12}C_{22})$ 
The determinant of $P_{A_{11}A_{22}}$ is a product of powers of $2$ and powers of $3$. Therefore $P_{A_{11}A_{22}}^{-1}$ is defined when the domain is restricted to a congruence subgroup with level $N$, which is a power of 2 times power of 3. It is invertible modulo $p$. Therefore, the condition whether we consider $C_{22}$ in $GL_2(\Z)$ or in $\Gamma_1(2,p)$ is independent of the conditions on $C_{12}$. Therefore,
$C_\Gamma(A)\backslash C_G(A)=\Gamma_1(2,p)\backslash GL_2(\Z)$ which has $p^2-1$ elements.
Then
$\chi(C_\Gamma(A))=(p^2-1)\chi(C_G(A))$.

Now, we return to the proof of the Theorem \ref{Gamma char}
We have
\begin{eqnarray*}
\chi_h(\Gamma,V)
=&&\sum_{\mbox{(B) in } \Gamma}\chi(C_\Gamma(B))Tr(B^{-1}|V)\\
=&&(p-1)\left(\sum_{\mbox{(A) in G,  multiplicity 1}}\chi(C_\Gamma(A))Tr(A|V)\right)\\
&&+(p^2-1)\left(\sum_{\mbox{(A) in G,  multiplicity 2}}\chi(C_\Gamma(A))Tr(A|V)\right)\\
=&&(p-1)\left(\sum^{(1)} R(A) \chi(C_\Gamma(A))Tr(A|V)\right)\\
&&+(p^2-1)\left(\sum^{(2)} R(A)\chi(C_\Gamma(A))Tr(A|V)\right)
\end{eqnarray*}
The first equality is by K. Brown formula. The second equality follows from the previous Lemma. The last equality follows from \cite{Thesis} Theorem 2.4.






\section{Euler characteristics of $\Gamma_1(3,p)$ with coefficients in any highest weight representation}

\begin{prop}
\label{prop Gamma1(3,p)}
We are going to sum over all conjugacy classes of torsion elements of $\Gamma_1(3,p)$ whose eigenvalue $1$ has multiplicity one. We will denote such summation over conjugacy classes in $G$
$\sum_{A}^{(1)}$.
Let $V$ be any finite dimensional highest weight representation
Then
\[
\sum^{(1)}_{(A)\in\Gamma_1(3,p)}\chi(C_{\Gamma_1(3,p)}(A))Tr(A^{-1}|V)
=
\frac{1}{2}\varphi(p)\chi_h(SL_3(\Z),V).
\]
\end{prop}
\proof
From Lemmas \ref{N_G} and \ref{Gamma center}, we know that
\[
\sum^{(1)}_{(A) \in\Gamma_1(3,p)}\chi(C_{\Gamma_1(3,p)}(A))Tr(A^{-1}|V)
=
(p-1)
\sum^{(1)}_{(A)\in GL_3(\Z)}\chi(C_{GL_3(\Z)}(A))Tr(A|V)
\]
There is a comparison between cohomology of $SL_3(\Z)$ and cohomology of $GL_3(\Z)$.
\begin{lemma}
\label{lemma SL}
Let $V$ be a finite dimensional representation of $GL_m(\Z)$ and let $det$ be the determinant representation.
Then
 
(a) $H^q(SL_m(\Z),V)=H^q(GL_m(\Z),V)+H^q(GL_m(\Z),V\otimes det)$;

(b) $\chi_h(SL_m(\Z),V)=\chi_h(GL_m(\Z),V)+\chi_h(GL_m(\Z),V\otimes det)$;

(c) If $m$ is odd and the representation $V$ is even, meaning $-I_m$ acts trivially, then
$\chi_h(SL_m(\Z),V)=\chi_h(GL_m(\Z),V)$ and $\chi_h(GL_m(\Z),V\otimes det)=0$;

(d) If $m$ is odd and the representation $V$ is even, meaning $-I_m$ acts trivially, then
$\chi_h(SL_m(\Z),V)=\chi_h(GL_m(\Z),V\otimes det)$ and $\chi_h(GL_m(\Z),V)=0$.
\end{lemma}
\proof
(a) Restrict $V$ to a representation of $SL_m$. Then induce the resulting restriction of $GL_m$. It becomes $Ind(V)=V + V\otimes det$.
Therefore, 
$H^q(SL_m(\Z),V)=H^q(GL_m(\Z),Ind(V))=H^q(GL_m(\Z),V)+H^q(GL_m(\Z),V\otimes det)$.

Part (b) is a direct consequence of part (a).

For part (c), if $-I_m$ acts trivially on $V$. Then $-I_m$ acts as $-1$ on $V\otimes det$. Therefore, 
$\chi_h(GL_m(\Z),V\otimes det)=0$. Then part (c) follows directly from part (b).
Part (d) is similar to part (c), where $V$ and $V\otimes det$ interchange their role.

From Lemma \ref{lemma SL} parts (c) and (d), and Theorem 1, we have the following.
\begin{lemma} If $V$ be a finite dimensional representation of $GL_3(\Z)$ then
\[\chi_h(SL_3(\Z),V)=2\sum_{det(A)=1}R(A)\chi(C_{GL_3(\Z)}Tr(A|V),\]
where the sum is over block diagonal matrices $A$ such that one of the diagonal entry is $1$ and the other diagonal block is a $2\times 2$ matrix from the set
$\{-I_2,T_3,T_4,T_6\}$. 
\end{lemma}

\proof (of the Proposition \ref{prop Gamma1(3,p)})
We have that
\begin{eqnarray*}
&&\sum^{(1)}_{(A)\in\Gamma_1(3,p)}\chi(C_{\Gamma_1(3,p)}(A))Tr(A^{-1}|V)\\
&&=
(p-1)\sum^{(1)}_{(A)\in GL_3(\Z)}\chi(C_{GL_3(\Z)}(A))Tr(A^{-1}|V)\\
&&= \frac{1}{2}\chi_h(SL_3(\Z),V).
\end{eqnarray*}
The first equality follows from Theorem \ref{Gamma char}
and the second from the previous Lemma.
\qed

Then we have the following.
\begin{thm} If $V$ be a finite dimensional representation of $GL_3(\Z)$, then
\begin{eqnarray*}
\chi_h(\Gamma_1(3,p),V)
=&&\frac{1}{2}(p-1)\chi_h(SL_3(\Z),V)\\
&&-\frac{1}{12}(p^2-1)Tr([-1,I_2]|V)
\end{eqnarray*}
\end{thm}
\proof
By Theorem \ref{Gamma char} we have that
\begin{eqnarray*}
\chi_h(\Gamma_1(3,p),V)
=&&(p-1)\sum^{(1)}R(A)\chi(C_{GL_3(\Z)})Tr(A|V)+ \\
&&+(p^2-1)\sum^{(2)}R(A)\chi(C_{GL_3(\Z)})Tr(A|V).
\end{eqnarray*}
From the Proposition we have $\sum^{(1)}R(A)\chi(C_{GL_3(\Z)})Tr(A|V)=\chi_h(SL_3(\Z),V)$.
For the sum $\sum^{(2)}$ we have
That the summation is is over one element, namely $A=[-1,I_2]$, For that element $R({-1,I_2})=(1-(-1))^2=4$ and $C_{GL_3(\Z)}([-1,I_2])=C_{GL_1(\Z)}(-1)\times C_{GL_2(\Z)}(I_2)=\Z_2\times GL_2(\Z)$. Therefore, $\chi(C_{GL_3(\Z)}([-1,I_2])=\frac{1}{2}\left(-\frac{1}{24}\right).$
Therefore,
\[\sum^{(2)}R(A)\chi(C_{GL_3}(\Z))Tr(A|V)=-\frac{1}{12}Tr([-1,I_2]|V).\]

We have computed $\chi_h(SL_3(\Z),V_\lambda)$ in \cite{BHHM}. 
\begin{thm}
Let $\lambda=(k,l,n)$ be a highest weight with $t^k$, $t^l$, and $t^n$ the characters on the split Cartan group of $GL_3$
If $V_\lambda$ is the highest weight representation then
\[\chi_h(SL_3(\Z),V_{[2a,2b]})=
-(1+\dim S_{2a-2b+2}+\dim S_{2b+2})\]

\[\chi_h(SL_3(\Z),V_{[2a+1,2b+1]})=
\dim S_{2a+4}-\dim S_{2b+2}\]

\[\chi_h(SL_3(\Z),V_{[2a,2b+1]})=
0\]

\[\chi_h(SL_3(\Z),V_{[2a+1,2b]})=
\dim S_{2a+4}-\dim S_{2b+2}\]
\end{thm}

In order to obtain exact values for  $\chi_h(\Gamma_1(3,p),V_\lambda)$, we need to compute the remaining term
$Tr(-1,I_2|V_\lambda)$. 
We are going to use Weyl trace formula for $GL3$. Let us recall it.
Let $H_a(A)=Tr(A|Sym^a)=Tr(A|V_{[a,0,0]})$

If $V_\lambda=V_{[a,b,0]}$ is a highest weight representation, then
\[Tr(A|V_{[a,b,0]})=
det
\left[
\begin{array}{ll}
H_a & H_{a+1}\\
\\
H_{b-1} & H_{b}
\end{array}
\right].
\]
We will consider four cases for the pair $a,b$, depending on their parity.

First 
\[
\begin{array}{rl}
H_a(-1,I_2)&=Tr([-1,I_2]|Sym^aV_3)\\&=\sum_{k=0}^a Tr(-1|Sym^aV_1)Tr(I_2|Sym^{a-k}V_2)\\ 
&=\sum_{k=0}^a (-1)^k (a-k+1).
\end{array}
\]
Therefore,
\[H_{2a}=a+1\]
and
\[H_{2a+1}=-(a+1).\]
Then 
\[\begin{array}{lll}
H_{2a,2b}=det
\left[
\begin{array}{ll}
H_{2a} & H_{2a+1}\\
\\
H_{2b-1} & H_{2b}
\end{array}
\right]
=
det
\left[
\begin{array}{ll}
a+1 & -(a+1)\\
\\
-b & b+1
\end{array}
\right]\\
=
(a+1)(b+1)-b(a+1)\\
=a+1.
\end{array}\]

\[\begin{array}{lll}
H_{2a+1,2b}=
det
\left[
\begin{array}{ll}
H_{2a+1} & H_{2a+2}\\
\\
H_{2b-1} & H_{2b}
\end{array}
\right]
=
det
\left[
\begin{array}{ll}
-(a+1) & a+2\\
\\
-b & b+1
\end{array}
\right]\\
=
-(a+1)(b+1)+b(a+2)\\
=b-a-1.
\end{array}\]

\[\begin{array}{lll}
H_{2a,2b+1}=
det
\left[
\begin{array}{ll}
H_{2a} & H_{2a+1}\\
\\
H_{2b} & H_{2b+1}
\end{array}
\right]
=
det
\left[
\begin{array}{ll}
a+1 & -(a+1)\\
\\
b+1 & -(b+1)
\end{array}
\right]\\
=
-(a+1)(b+1)+(a+1)(b+1)\\
=0.
\end{array}\]

\[\begin{array}{lll}
H_{2a+1,2b+1}=
det
\left[
\begin{array}{ll}
H_{2a+1} & H_{2a+2}\\
\\
H_{2b} & H_{2b+1}
\end{array}
\right]
=
det
\left[
\begin{array}{ll}
-(a+1) & a+2\\
\\
b+1 & -(b+1)
\end{array}
\right]\\
=
(a+1)(b+1)-(a+2)(b+1)\\
=-b-1.
\end{array}\]

We also need their twists by the representation $det$. Since $det(-1,I_2)=1$. We obtain that
\[H_{a+1,b+1,1}=-H_{a,b,0}=-H_{a,b}.\]
For completeness, we compile the formulas for $H_{a,b,c}$ is the following prosition
\begin{prop}
(a) $H_{2a,2b}=a+1.$

(b) $H_{2a+1,2b+1,1}=-(a+1).$

(c) $H_{2a+1,2b}=-a+b-1.$

(d) $H_{2a+2,2b+1,1}=a-b+1.$

(e) $H_{2a+1,2b+1}=-(b+1).$

(f) $H_{2a+2,2b+2,1}=b+1.$

(g) $H_{2a,2b+1}=0.$

(h) $H_{2a+1,2b+2,1}=0.$

\end{prop}






\begin{thm}  (Homological Euler characteristics of $\Gamma_1(3,p)$)

(a) $\chi_h(\G_1(3,p),V_{2a,2b,0})
=-\frac{p-1}{2}(1+\dim S_{2a-2b+2}+\dim S_{2b+2})
-\frac{p^2-1}{12}(a+1)
$

(b) $\chi_h(\G_1(3,p),V_{2a+1,2b+1,1})=
-\frac{p-1}{2}(1+\dim S_{2a-2b+2}+\dim S_{2b+2})
+\frac{p^2-1}{12}(a+1)
$

(c) $\chi_h(\G_1(3,p),V_{2a+1,2b,0})=
\frac{p-1}{2}(\dim S_{2a+4}-\dim S_{2b+2})
-\frac{p^2-1}{12}(a-b+1)
$

(d) $\chi_h(\G_1(3,p),V_{2a+2,2b+1,1})=
\frac{p-1}{2}(\dim S_{2a+4}-\dim S_{2b+2})
+\frac{p^2-1}{12}(a-b+1)
$

(e) $\chi_h(\G_1(3,p),V_{2a+1,2b+1,0})=
\frac{p-1}{2}(\dim S_{2a+4}-\dim S_{2a-2b+2})
+\frac{p^2-1}{12}(b+1)
$

(f) $\chi_h(\G_1(3,p),V_{2a+2,2b+2,1})=
\frac{p-1}{2}(\dim S_{2a+4}-\dim S_{2a-2b+2})
-\frac{p^2-1}{12}(b+1)
$

(g) $\chi_h(\G_1(3,p),V_{2a,2b+1,0})=0$

(h) $\chi_h(\G_1(3,p),V_{2a+1,2b+2,1})=0$
\end{thm}




\section{Cohomology of $\Gamma_1(3,p)$ with trivial and with determinant coefficients}

Form the last Theorem we have
\begin{cor}

(a) $\chi_h(\G_1(3,p),\C)
=-\frac{p^2-1}{12}-\frac{p-1}{2}
$

(b) $\chi_h(\G_1(3,p),det)
=+\frac{p^2-1}{12}-\frac{p-1}{2}
$
\end{cor}

By Corollary \ref{Gamma1(2,p) cusp},
we have
\[\dim H^1_{cusp}(\Gamma_1(2,p),\C)=1+\frac{1}{24}(p^2-1) -\frac{1}{2}(p-1)
\]
and 
\[\dim H^1_{cusp}(\Gamma_1(2,p),V_3)=\frac{2}{24}(p^2-1)-\frac{1}{2}(p-1).
\]

We will use the notation $\Gamma_1=\Gamma_1(3,p)$ and $\Gamma_1^+=\Gamma_1 \cap SL_3(\Z)$.
Then
\[H^q(\Gamma_1^+,\C)=H^q(\Gamma_1,\C)+H^q(\Gamma_1,det),\]
since the induced representation of the trivial coefficients from $\Gamma_1^+$ to $\Gamma_1$ is $\C\oplus det$.
From results by Lee and Schwermer \cite{LSch}, we have

\begin{thm}

$H^q(\Gamma_1^+,\C)=
\left\{
\begin{tabular}{ll}
$S_3(\Gamma_1(2,p))+\C^{p-1}$			&	$q=2$,\\
\\
$2S_2(\Gamma_1(2,p))+\C^{-1+\frac{p-1}{2}}$	&	$q=3$.
\end{tabular}
\right.
$
\end{thm}

\begin{cor}
$\dim H^q(\Gamma_1^+,\C)=
\left\{
\begin{tabular}{ll}
$1$						&	$q=0$,\\
\\
$0$						&	$q=1$,\\
\\
$\frac{p^2-1}{12}-\frac{p-1}{2}$	&	$q=2$,\\
\\
$\frac{p^2-1}{12}+\frac{p-1}{2}+1$	&	$q=3$.
\end{tabular}
\right.
$
\end{cor}
\proof
From Theorems 7, we have that the dimensions of the space of cusp forms of weight $2$ and $3$ are given by 
\[\dim S_2(\Gamma_1(2,p))=\dim H^1_{cusp}(\Gamma_1(2,p),\C)=1+\frac{1}{24}\varphi_2(p) -\frac{1}{2}\varphi(p)
\]
\[dim S_3(\Gamma_1(2,p))=\dim H^1_{cusp}(\Gamma_1(2,p),V_2)=\frac{1+1}{24}\varphi_2(p)-\frac{1}{2}\varphi(p).
\]
Then the Corollary follows directly from the previous Theorem.

\begin{thm}
(a) $\dim H^q(\Gamma_1(3,p),\C)
=
\left\{
\begin{tabular}{ll}
$1$								& $q=0$,\\
$\frac{p^2-1}{12}+\frac{p-1}{2}+1$ 		& $q=3$,\\
$0$								& $q\neq 0,3$.
\end{tabular}
\right.
$

(b) $\dim H^q(\Gamma_1(3,p),det)
=
\left\{
\begin{tabular}{ll}
$\frac{p^2-1}{12}-\frac{p-1}{2}$ 	& $q=2$,\\
$0$						& $q\neq 2$.
\end{tabular}
\right.
$
\end{thm}
\proof
The homological Euler characteristic of $\Gamma_1=\Gamma_1(3,p)$ with trivial coefficients is $-\frac{p^2-1}{12}-\frac{p-1}{2}$. We also have $H^0=\C$ and $H^1=0$. Therefore the only possibility is to have $H^2(\Gamma_1,\C)=0$.
Similarly, the only possibility for the cohomology of $\Gamma_1(3,p)$ with determinant coefficients is to have $\dim H^3(\Gamma_1,det)=0$. \qed

Note that Goncharov \cite{G1} obtained formulas for the cohomology of $\Gamma_1^+=\Gamma_1(3,p)\cap SL_3(\Z)$ but not the cohomology of $\Gamma_1=\Gamma_1(3,p)$. The reason is that his computation followed the method presented in \cite{LSch} that works for congruence subgroups of $SL_3(\Z)$, not for congruence subgroups of $GL_3(\Z)$.




\section{Cohomology of $GL_4(\Z)$ with trivial and with determinant coefficients}

Now, we are going to compute the cohomology of $GL_4(\Z)$ with coefficients in determinant representation.

In order to do that, first we need the following result
by Elbas-Vincent, Gangle and Soule  \cite{EVHS}.
\begin{thm}
\label{SL4}
\[
H^i(SL_4(\Z),\C)=
\left\{
\begin{tabular}{lll}
$\C$ & for $i=0$ or $3$,\\
\\
$0$ & otherwise.
\end{tabular}
\right.
\]
\end{thm}

From Lemma \ref{lemma SL}, we know that
\[\dim[H^i(GL_4(\Z),det)]=  \dim[H^i(SL_4(\Z),\C)]  -  \dim[H^i(GL_4(\Z),\C)].\]
We need to examine $\dim[H^i(GL_4\Z),\C)].$

\begin{prop}
\[
H^i(GL_4(\Z),\C)=
\left\{
\begin{tabular}{lll}
$\C$ & for $i=0$,\\
\\
$0$ & otherwise.
\end{tabular}
\right.
\]
\end{prop}
\proof
Obviously, $H^0(GL_4(\Z),\C)=\C.$ Also,
by Theorem \ref{SL4}, we have that 
$H^i(GL_4(\Z),\C)=0$ for $i\neq 0, 3$ and $H^3(GL_4(\Z),\C)=0$ or $\C$, since $H^i(GL_4(\Z),\C)$ is a direct summand of $H^i(SL_4(\Z),\C)=0$.
There are only two possibilities $H^3(GL_4(\Z),\C)=0$ or $\C$. 
We are going to show that $H^3(GL_4(\Z),\C)=\C$ leads to a contradiction. It will imply that $H^3(GL_4(\Z),det)=\C$.
Let $\sum^+$ and $\sum^-$ be summations over torsion classes of matrices where $+$ signifies that the determinant is $+1$ and $-$ signifies that the determinant is $-1$.
Then
\[\chi_h(GL_4(\Z),\C)=\sum^+\chi(C(A) + \sum^-\chi(C(A))\]
and
\[\chi_h(GL_4(\Z),det)=\sum^+\chi(C(A) - \sum^-\chi(C(A)).\]
Therefore
\[\chi_h(GL_4(\Z),\C)-\chi_h(GL_4(\Z),det)=2 \sum^-\chi(C(A)).\]

In the summation $\sum^-$ the centralizers as algebraic groups are $GL_1\times GL_1\times R_{K/\Q}GL_1$, where the extension $K$ is an imaginary quadratic field.
The $\Z$ point of such a group form a finite group. Therefore, the orbifold Euler characteristic of each of the centralizers is positive. As a result,
\[\chi_h(GL_4(\Z),\C)-\chi_h(GL_4(\Z),det)>0.\]
However, by assumption, $H^3(GL_4(\Z),\C)=\C$. Then $\chi_h(GL_4(\Z),\C)=0$ and $\chi_h(GL_4(\Z),det)=0$. This leads to a contradiction. We obtain the following
The only other possibility is  $H^3(GL_4(\Z),\C)= 0$,
which is the statement of the Proposition.

Since 
$\dim[H^i(GL_4(\Z),det)]=  \dim[H^i(SL_4(\Z),\C)]  -  \dim[H^i(GL_4(\Z),\C)],$
we obtain the following.
\begin{cor}
\[
H^i(GL_4(\Z),det)=
\left\{
\begin{tabular}{lll}
$\C$ & for $i=3$\\
\\
$0$ & otherwise
\end{tabular}
\right.
\]
\end{cor}

\section{Euler Characteristics of $\Gamma_1(4,p)$ with trivial and with determinant coefficients}

First, we will prove a general formula for the homological Euler characteristics of $\Gamma_1(4,p)$ with coefficients in any finite dimensional highest weight representation.
Let $\sum^{(k)}$ denotes summation over block diagonal matrices $A$ in $GL_4(\Z)$ (up to conjugation) with $\chi(C_{GL_4(\Z)}(A))\neq 0$, whose eigenvalue $+1$ has multiplicity $k$, and
let 
\[\Sigma^{(k)}(V)=\sum^{(k)}R(A)\chi(C_{GL_4(\Z)}(A))Tr(A|V).\]

Now we are going to compute exactly homological Euler characteristics of $\Gamma_1(4,p)$ with trivial coefficients.

\begin{lemma}

(a) $\Sigma^{(1)}(\C)=1$

(b)  $\Sigma^{(1)}(det)=-1$

(c) $\Sigma^{(2)}(\C)=-\frac{1}{12}$

(d)  $\Sigma^{(2)}(det)=-\frac{1}{12}$
\end{lemma}
\proof
Part (a).  We have 
\begin{eqnarray*}
\Sigma^{(1)}(\C)=	&&R(1,-1,T_3)\chi(C(1,-1,T_3))+ R(1,-1,T_4)\chi(C(1,-1,T_4))+\\
				&&+R(1,-1,T_6)\chi(C(1,-1,T_6)).
\end{eqnarray*}
We have $R(1,-1,T_3)=(1-(-1))(1-w^2)(1-w^{-2})(-1-w^2)(-1-w^{-2})=6$
Similarly, $R(1,-1,T_6)=6$.
And $R(1,-1,T_4)=(1-(-1))(1-i)(1+i)(-1-i)(-1+i)=8$.
$C(1,-1,T_3)=\Z_2\times \Z_2\times \Z_6$ and $\chi(C(1,-1,T_3))=\frac{1}{24}$,
$C(1,-1,T_6)=\Z_2\times \Z_2\times \Z_6$ and $\chi(C(1,-1,T_6))=\frac{1}{24}$,
and
$C(1,-1,T_4)=\Z_2\times \Z_2\times \Z_4$ and $\chi(C(1,-1,T_4))=\frac{1}{16}$.
Then 
$\Sigma^{(1)}(\C)=\frac{6}{24}+\frac{6}{24}+\frac{8}{16}=1$

Part (b). Since the matrices under the summation have determinant $-1$, we obtain that 
$\Sigma^{(1)}(det)=-\Sigma^{(0)}(\C)=-1$.

Part (c). We have 
\begin{eqnarray*}
\Sigma^{(2)}(\C)=	&&R(I_2,T_3)\chi(C(I_2,T_3))+ R(I_2,T_6)\chi(C(I_2,T_6))+\\
				&&+R(I_2,T_4)\chi(C(I_2,T_4))\\
				&&+R(I_2,-I_2).
\end{eqnarray*}
We have 

$R(I_2,T_3)=(1-w^2)^2(1-w^{-2})^2=9$
Similarly, 

$R(I_2,T_6)=(1-w)^2(1-w^{-1})^2=1$
where $w$ is primitive $6$-th root of 1.
Also 
$R(I_2,T_4)=(1-i)^2(1+i)^2=4$.
and
$R(I_2,-I_2)=2^4$
For the centralizers we have
$C(I_2,T_3)=GL_2(\Z)\times \Z_6$ and $\chi(C(I_2,T_3))=-\frac{1}{24}\cdot\frac{1}{6}$,
$C(I_2,T_6)=GL_2(\Z)\times \Z_6$ and $\chi(C(I_2,T_6))=-\frac{1}{24}\cdot\frac{1}{6}$,
$C(I_2,T_4)=GL_2(\Z)\times \Z_4$ and $\chi(C(I_2,T_4))=-\frac{1}{24}\cdot\frac{1}{4}$
and
$C(I_2,-I_2)=GL_2(\Z)\times GL_2(\Z)$ and $\chi(C(I_2,-I_2))=\frac{1}{24}\cdot\frac{1}{24}$
and
Then 
$\Sigma^{(2)}(\C)=-\frac{9}{144}-\frac{1}{144}-\frac{1}{24}+\frac{16}{24^2}=-\frac{1}{12}$

Part (d). Since the matrices under the summation have determinant $+1$, we obtain that 
$\Sigma^{(2)}(det)=\Sigma^{(2)}(\C)=-\frac{1}{12}$.
\qed

From Theorem \ref{Gamma char} we obtain
\begin{thm}
 \[\chi_h(\Gamma_1(4,p),\C)=(p^2-1)\Sigma^{(2)}(\C)+(p-1)\Sigma^{(1)}(\C)=-\frac{1}{12}(p^2-1)+(p-1)\]
 and
 \[\chi_h(\Gamma_1(4,p),det)=(p^2-1)\Sigma^{(2)}(det)+(p-1)\Sigma^{(1)}(det)=-\frac{1}{12}(p^2-1)-(p-1).\]
\end{thm}

\section{Euler Characteristics of $\Gamma_1(5,p)$ with trivial and with determinant coefficients}

\begin{thm}
Let $V$ be a finite dimensional highest weight representation of $GL_5$.

(a) If $-I_5\in GL_5(\Z)$ acts trivially on $V$ then
\[
\chi_h(\Gamma_1(5,p),V)=[(p-1)+(p^2-1)]\chi_h(SL_5(\Z),V)-[p^2-1]\Sigma^{(1)}(V).
\]
(a) If $-I_5\in GL_5(\Z)$ acts on $V$ as multiplication by $-1$ then
\[
\chi_h(\Gamma_1(5,p),V)=[(p-1)-(p^2-1)]\\chi_h(SL_5(\Z),V)-[p^2-1]\Sigma^{(1)}(V),
\]
where $\Sigma^{(1)}(V)$ is the sum of $R(A)\chi(C_{GL_5(\Z)}(A)Tr(A|V)$ over a set of three elements, namely,
$[-1,T_3,T_6]$, $[-1,T_3,T_4]$ and $[-1,T_4,T_6]$.
More over if $V=V_\lambda$ then $\Sigma^{(1)}(V_\lambda)$ is bounded and periodic with respect to the wight $\lambda$.
\end{thm}
\proof Let $\Sigma^{(k)}(V_\lambda)$  is the sum of $R(A)\chi(C_{GL_5(\Z)}(A)Tr(A|V)$ over the set of block diagonal matrices having eigenvalue $+1$  with multiplicity $k$.
Part (a).  If $A$ has the eigenvalue $+1$ with multiplicity one then  $-A$ had the eigenvalue with multiplicity either zero or two. 
Since $-I_5$ acts trivially, we have that  $\Sigma^{(1)}(V) = \Sigma^{(0)}(V) + \Sigma^{(2)}(V)$.
By Theorem \ref{Gamma char} we have
\begin{eqnarray*}
\chi_h(SL_5(\Z),V)
=&&\chi_h(GL_5(\Z),V)+\chi_h(GL_5(\Z),V\otimes det)\\
=&&
\left(
\Sigma^{(0)}(V) + \Sigma^{(1)}(V) + \Sigma^{(2)}(V)
\right)\\
&&+
\left(
-\Sigma^{(0)}(V) + \Sigma^{(1)}(V) -\Sigma^{(2)}(V)
\right)\\
&&=
2 \Sigma^{(1)}(V) 
\end{eqnarray*}
Then,
\begin{eqnarray*}
\chi_h(\Gamma_1(5,p),V)
&&= (p^2-1)\Sigma^{(2)}(V) + (p-1)\Sigma^{(1)}(V)\\
&&= (p^2-1)\left(\Sigma^{(1)}(V)-\Sigma^{(0)}(V)\right) + (p-1)\Sigma^{(1)}(V)\\
&&= [(p^2-1)+(p-1)\Sigma^{(1)}(V)- (p^2-1)\Sigma^{(0)}(V)\\
&&= \frac{1}{2}[(p^2-1)+(p-1)]\chi_h(SL_5(\Z),V)- (p^2-1)\Sigma^{(0)}(V)\\
\end{eqnarray*}

Part(b). Since $-I_5$ acts multiplication by $-1$, we have that  $\Sigma^{(1)}(V) = -\Sigma^{(0)}(V) - \Sigma^{(2)}(V)$.
Then,
\begin{eqnarray*}
\chi_h(\Gamma_1(5,p),V)
&&= (p^2-1)\Sigma^{(2)}(V) + (p-1)\Sigma^{(1)}(V)\\
&&= (p^2-1)\left(-\Sigma^{(1)}(V)-\Sigma^{(0)}(V)\right) + (p-1)\Sigma^{(1)}(V)\\
&&= [(p^1-1)+(p^2-1)\Sigma^{(1)}(V)- (p^2-1)\Sigma^{(0)}(V)\\
&&= \frac{1}{2}[(p-1)-(p^2-1)]\chi_h(SL_5(\Z),V)- (p^2-1)\Sigma^{(0)}(V)\\
\end{eqnarray*}
From by Elbas-Vincent, Gangle and Soule, see \cite{EVHS}, we know that
\begin{thm}
$H^q(SL_5(\Z),\C)
=
\left\{
\begin{tabular}{ll}
$\C$	&	if $q=0,5$\\
$0$	&	otherwise
\end{tabular}
\right.
$
\end{thm}

Then
\begin{cor}
$\chi_h(GL_5(\Z),\C)=\chi_h(SL_5(\Z),\C)=0$
\end{cor}

It remains to compute $\Sigma^{(1)}(\C)$. It consists of a sum over the three elements $(1,T_3,T_6)$, $(1,T_3,T_4)$ and $(1,T_6,T_4)$
For the resultants we have

$R(1,T_3,T_6)=R(1,T_3)R(1,T_6)R(T_3,T_6)=3\cdot 1\cdot 4=12$,

$R(1,T_3,T_4)=R(1,T_3)R(1,T_4)R(T_3,T_4)=3\cdot 2\cdot 1 = 6$,

$R(1,T_6,T_4)=R(1,T_6)R(1,T_4)R(T_6,T_4)=1\cdot 2\cdot 1 = 2$.

For the corresponding centralizers we have

$C(1,T_3,T_6)=\Z_2\times \Z_6\times \Z_6$ and $\chi(C(1,T_3,T_6))=\frac{1}{72}$,

$C(1,T_3,T_4)=\Z_2\times \Z_6\times \Z_4$ and $\chi(C(1,T_3,T_4))=\frac{1}{48}$,

$C(1,T_6,T_4)=\Z_2\times \Z_6\times \Z_4$ and $\chi(C(1,T_6,T_4))=\frac{1}{48}$.

Therefore
$\Sigma^{(1)}(\C)=\frac{12}{72}+\frac{6}{48}+\frac{2}{48}=\frac{1}{3}$.
We obtain
\begin{thm}
(a) $\chi_h(\Gamma_1(5,p),\C)=-\frac{1}{3}(p^2-1);$

(b) $\chi_h(\Gamma_1(5,p),det)=-\frac{1}{3}(p^2-1).$
\end{thm}
\section*{Acknowledgements}
I would like to thank G\"unter Harder, Jitendra Bajpai, Moya Giusti and Alexander Goncharov for bringing me back to the topic of arithmetic groups even though all that turned cold at some point. I would also like to thank Paul Gunnells and  Philippe Elbaz-Vincent  for the interest in my work. 

This work was supported by PSC-CUNY cycle 55 grant.

\end{document}